\documentclass[a4paper]{article}

\pdfoutput=1

\usepackage[a4paper]{geometry}

\usepackage{amsmath}
\usepackage{graphicx}

\usepackage{amsfonts}
\usepackage{mathtools}
\usepackage{stmaryrd}
\usepackage{bussproofs}
\usepackage{tikz-cd}

\usepackage{mymacros}

\usepackage{amsthm}

\usepackage{hyperref}

\hypersetup{
    colorlinks=true,
    linkcolor=blue,
    citecolor=blue,
    filecolor=blue,
    urlcolor=blue,
}

\makeatletter
\renewenvironment{proof}[1][\proofname]{\par
  \pushQED{\qed}
  \normalfont \topsep6\p@\@plus6\p@\relax
  \trivlist
  \item\relax
  {\itshape #1\@addpunct{.}}
  \hspace\labelsep\ignorespaces
  }
  {\popQED\endtrivlist\@endpefalse}
\makeatother

\newtheorem*{theorem}{Theorem}
\newtheorem{therm}{Theorem}
\newtheorem{lemma}[therm]{Lemma}
\newtheorem{proposition}[therm]{Proposition}

\newtheorem{definition}[therm]{Definition}
\theoremstyle{definition}
\newtheorem{example}[therm]{Example}
\theoremstyle{remark}
\newtheorem{remark}[therm]{Remark}

\title{The Scott model of PCF in univalent type theory}
\author{Tom de Jong \vspace{.2cm} \\
  \small{University of Birmingham, Birmingham, United Kingdom} \\
  \small{Email: \texttt{t.dejong@pgr.bham.ac.uk}}}
\date{\small{\today}}

\begin{document}

\maketitle

\begin{abstract}
  We develop the Scott model of the programming language PCF in univalent type
  theory. Moreover, we work constructively and predicatively. To account for the
  non-termination in PCF, we use the lifting monad (also known as the
  partial map classifier monad) from topos theory, which has been extended to
  univalent type theory by Escard\'o and Knapp. Our results show that lifting is
  a viable approach to partiality in univalent type theory. Moreover, we show
  that the Scott model can be constructed in a predicative and constructive
  setting. Other approaches to partiality either require some form of choice or
  quotient inductive-inductive types. We show that one can do without these
  extensions.
\end{abstract}

\section{Introduction}\label{sec:introduction}
We develop the Scott model of the programming language PCF in constructive
predicative univalent mathematics. In 1969, Dana Scott~\cite{Scott1993} proposed a
logic (LCF) for computing with functionals. In 1977, Gordon Plotkin~\cite{Plotkin1977}
considered LCF as a programming language (PCF); introducing operational
semantics based on Scott's logic and proving (and formulating) soundness and
computational adequacy. Later, the techniques of Scott and Plotkin were extended
to many other programming languages~\cite{Plotkin1983}. These developments all
took place in (informal) set theory with classical logic.

Our aim is to test these techniques in Voevodsky's constructive univalent type
theory~\cite{HoTTbook}. Our development differs from the classical
approach~\cite{Streicher2006} in three key ways. First of all, we have
situated our development in the framework of univalent mathematics. Secondly,
our work takes place in a constructive meta-theory. Thirdly, we work
predicatively (meaning we do not assume propositional resizing).

The essential difference (for our development) between univalent type theory on
the one hand, and set theory or systems like Coq on the other, is the treatment
of truth values (propositions). We will discuss manifestations of this
difference in Section~\ref{univalentdifferences} and throughout the paper.

\subsection{Technical preliminaries}
In this section we briefly explain the syntax of PCF and its computational
behaviour. Moreover, we recall the notion of denotational semantics and the
Scott model of PCF (in a classical setting) in particular. We also mention two
fundamental properties that a model of PCF should enjoy: soundness and
computational adequacy. Finally, we recall the lifting monad in the context of
univalent type theory and sketch the construction of the Scott
model in constructive univalent type theory.

\subsubsection{PCF}\label{background:PCF}
PCF~\cite{Plotkin1977} is a typed programming language. A detailed description
of PCF is given in Section~\ref{sec:PCF}. We briefly discuss its most characteristic
features here. PCF is a typed \(\lambda\)-calculus with additional
constants. For example, we have numerals \(\underline{n}\) of base type
\(\iota\) corresponding to natural numbers and basic operations on them, such as
a predecessor term \(\pred\) and a term \(\ifz\) that allows us to
perform case distinction on whether an input is zero or not. The most striking
feature of PCF is its fixed point combinator \(\fix_\sigma\) for every
PCF type \(\sigma\). The idea is that for a term \(t\) of function type
\(\sigma \Rightarrow \sigma\), the term \(\fix_\sigma\,t\) of type
\(\sigma\) is a fixed point of \(t\). The use of \(\fix\) is that it
gives us general recursion.

The operational semantics of PCF is a reduction strategy that allows us to
compute in PCF. We write \(s \smallstep t\) for \(s\) reduces to \(t\). We show a
few examples below:
\[
  \pred \underline{0} \smallstep \underline{0}; \quad
  \pred \underline{n+1} \smallstep \underline{n}; \quad
  {\ifz s\, t\, \underline{0}} \smallstep s; \quad
  {\ifz s\, t\, \underline{n+1}} \smallstep t; \quad
  \fix f \smallstep f(\fix f).
\]
We see that \(\pred\) indeed acts as a predecessor function and that \(\ifz\)
performs case distinction on whether its third argument is zero or not. The
reduction rule for \(\fix\) reflects that \(\fix f\) is a fixed point \(f\) and
may be seen as an unfolding (of a recursive definition).

As an example of the use of \(\fix\), consider a function \(g\) on the natural
numbers given by the recursive definition: \(g(0) \coloneqq s\) and
\(g(n+1) \coloneqq t(g(n))\). We can define \(g\) in PCF as \(\fix G\) where
\(G \coloneqq \lambda (f : \iota \Rightarrow \iota).\lambda (x:\iota).\ifz
s\,(t(\pred x))\,x\). Having general recursion also introduces non-termination,
as for example the successor function on naturals has no fixed point.

Instead of the formulation by Plotkin~\cite{Plotkin1977}, which features variables
and \(\lambda\)-abstraction, we revert in Section~\ref{sec:PCF} to the original,
combinatory, formulation of the terms of LCF by Scott~\cite{Scott1993} in order
to simplify the technical development.

\subsubsection{Models of PCF}\label{background:models}
We have seen that the operational semantics give meaning to the PCF terms by
specifying computational behaviour. Another way to give meaning to the PCF terms
is through denotational semantics, i.e.\ by giving a model of PCF.
A model of PCF assigns to every PCF type $\sigma$ some mathematical structure
$\densem{\sigma}$ and to every PCF term $t$ of type $\sigma$ an element
$\densem{t}$ of $\densem{\sigma}$.

\paragraph{Soundness and computational adequacy.} Soundness and computational
adequacy are important properties that a model of PCF should have.

Soundness states that if a PCF term $s$ computes to $t$ (according to the
operational semantics), then their interpretations are equal in the model
(symbolically, $\densem{s} = \densem{t}$).

Computational adequacy is completeness at the base type \(\iota\). It says that
for every term $t$ of type $\iota$ and every natural number $n$, if
$\densem{t} = \densem{\underline n}$, then $t$ computes to $\underline n$.

\paragraph{The Scott model, classically.}
To model PCF and its non-termination, Dana Scott~\cite{Scott1993} introduced the Scott
model: a type is interpreted as a directed complete poset with a least element
(or dcpo with $\bot$, for short). Concretely, PCF types are interpreted as
follows.

\paragraph*{Interpreting the base type $\iota$.}  One proves that adding a
least element $\bot_{\mathbb N}$ to the set $\mathbb N$ of natural numbers
yields a dcpo with $\bot$, known as the \emph{flat natural numbers}. This is
then the interpretation of the base type $\iota$. This least element
$\bot_{\mathbb N}$ serves as the interpretation of a term of type $\iota$ that
does not compute to a numeral, like \(\fix\succc\) where \(\succc\) denotes the
successor map on \(\iota\).

\paragraph*{Interpreting function types.}  Function types are interpreted by
considering continuous maps (i.e.\ monotone maps that preserve directed suprema)
between two dcpos with $\bot$. Such maps can be ordered pointwise to form
another dcpo with $\bot$.

\paragraph*{}
A striking feature, and the crux of the Scott model, is that every continuous
map has a (least) fixed point. Moreover, the assignment of a continuous map to
its least fixed point is continuous. This allows us to soundly interpret the
characteristic $\fix$ term of PCF.

The Scott model was proved sound and computationally adequate
by Plotkin~\cite{Plotkin1977}. A modern presentation may be found in Streicher's~\cite{Streicher2006}.

\paragraph{Issues with constructivity.}
While the interpretation of function types goes through constructively, the
above interpretation of the base type \(\iota\) is problematic from a
constructive viewpoint. Indeed, the proof that the flat natural numbers form a
dcpo relies on classical reasoning in its analysis of the directed subsets:
excluded middle allows us to prove that every directed subset of the flat
natural numbers is exactly one of \(\{\bot\}\), \(\{\bot,n\}\) or \(\{n\}\) for
some natural number \(n\). In fact, we can show that this reliance is in some
sense essential: in Section~\ref{sec:constructiveissuespartiality} we prove that
if the flat natural numbers form a dcpo, then the Limited Principle of
Omniscience (LPO) holds. This principle asserts that every binary sequence is
either \(0\) everywhere or it attains the value \(1\) at some point. LPO is not
constructively acceptable~\cite[p.~9]{Bishop1967}, it is even provably false in
some varieties of constructive mathematics~\cite[pp.~3--4]{BridgesRichman1987},
and it is independent of Martin-L\"of Type Theory~\cite{Escardo2018}.

\subsubsection{Univalent type theory}\label{univalentdifferences}

As mentioned at the beginning of Section~\ref{sec:introduction}, an essential
difference between univalent type theory on the one hand, and set theory or
systems like Coq on the other, is the treatment of truth values
(propositions). To illustrate this difference, consider the definition of a
poset (cf.\ Definition~\ref{def:poset}).

\begin{example}\label{ex:posetuniv}
  In set theory, the mathematical structure is provided by a set $X$ and a
  binary relation $\leq$ on $X$. Moreover, this relation is required to be
  reflexive, transitive and antisymmetric. Reflexivity,
  $\forall_{x \in X} x \leq x$ is a logical statement that is bivalent.

  In type theory, if we define ${\leq} : {X \to X \to \mathsf{Type}}$, with
  $\mathsf{Type}$ some type universe, then the type encoding reflexivity,
  $\prod_{x : X} x \leq x$, may have more than one element. This is a
  fundamental difference with set theory.

  In Coq, we could instead define ${\leq} : {X \to X \to \texttt{Prop}}$, where
  $\texttt{Prop}$ is Coq's special sort of propositions. This sort is defined
  such that (for instance) reflexivity, \(\forall_{x : X}\,x \leq x\), is again
  in~\(\texttt{Prop}\).

  The crucial difference between these approaches and the univalent approach, is
  that in univalent type theory, we \emph{prove} that something is a proposition
  (truth value). Following Voevodsky, we define a type to be a proposition
  (truth value, subsingleton) if it has at most one element with respect to its
  identity type, i.e.\ up to propositional equality. To define posets, we then
  ask for a witness that the type $x \leq y$ is a proposition for every
  $x,y : X$. This allows us, in the presence of function extensionality (which
  is a consequence of the univalence axiom), to prove that reflexivity and
  transitivity are propositions. For example, for reflexivity, we wish to show
  that the type \(\prod_{x : X} x \leq x\) is a proposition. So let \(f,g\) be
  two elements of this type. By~function extensionality, it suffices to show
  that \(f(x) = g(x)\) for every \(x : X\). But the type of \(f(x)\) and
  \(g(x)\) is \(x \leq x\), which is a proposition by requirement, so \(f(x)\)
  and \(g(x)\) must be (propositionally) equal, as desired. Finally, we require
  $X$ to be a set: any two elements of $X$ are equal in at most one way. This
  ensures, using function extensionality again, that antisymmetry is a
  proposition.
\end{example}

Sometimes, we will want to make a type into a proposition, by identifying its
elements. This is achieved through the propositional truncation, a higher
inductive type. For example, we will need it to define directed families
(Definition~\ref{def:isdirected}), but also to define the reflexive transitive
closure of a proposition-valued relation
(Definition~\ref{def:refltransclos}). We will further explain these examples in
the main text. The universal property of the propositional truncation is
described in Section~\ref{sec:overview}. For more on propositions, sets and
propositional truncation in univalent type theory, see \cite[Chapter~3]{HoTTbook}.

\subsection{Overview of results}\label{sec:overview}
We work in intensional Martin-L\"of Type Theory with inductive types (including
the empty $\emptyt$, unit $\unitt$, natural numbers $\natt$, and identity
types), $+$-, $\Sigma$- and $\Pi$-types. As usual, we simply write \(x=y\) for
the identity type \(\Id_X(x,y)\), use \({\equiv}\) for the judgemental equality
and write \({\simeq}\) for Voevodsky's notion of type equivalence.

We need (at least) two universes $\mathcal U_{0}, \mathcal U_{1}$ closed under
$+$-, $\Sigma$- and $\Pi$-types, such that $\mathcal U_{0}$ contains $\emptyt$,
$\unitt$ and $\natt$, while $\mathcal U_1$ contains $\mathcal U_0$. We work
predicatively, i.e.\ we do not assume propositional resizing, so the type of
propositions in \(\mathcal U_0\), denoted by \(\Omega\), lives in the universe
\(\mathcal U_1\).

We also assume two extensionality axioms. The first is function extensionality,
which asserts that pointwise equal functions are equal. Given two (dependent)
functions \(f,g : \prod_{a : A}B(a)\), we write \(f \sim g\) for the type
\(\prod_{a : A} f(a) = g(a)\), often called the type of homotopies between \(f\)
and \(g\). Function extensionality makes the type \(f \sim g\) equivalent to the
identity type \(f = g\). The second is propositional extensionality, which says
that logically equivalent propositions are equal, i.e. if \(P\) and \(Q\) are
propositions, then \(P \leftrightarrow Q\) implies \(P = Q\). In the presence of
function extensionality this is equivalent to
\((P \leftrightarrow Q) \simeq (P = Q)\).

Although we do not need the univalence axiom at any point, we remark that both
extensionality axioms above follow from it. Moreover, we emphasise the
importance of the idea of truncation levels, which is fundamental to univalent
type theory.

Finally, we assume the existence of a single higher inductive type, the
propositional truncation: given a type \(X\) in a universe \(\mathcal U\), we
assume that we have a \emph{proposition} \(\squash*{X}\) in \(\mathcal U\)
with a map \(\tosquash*{-} : X \to \squash*{X}\) such that if \(P\) is a
proposition in any universe and \(f : X \to P\) is a map, then \(f\) factors
through \(\tosquash*{-}\).  Diagrammatically,
\[
  \begin{tikzcd}
    X \ar[dr, "\tosquash*{-}"'] \ar[rr, "f"] & & {P} \\
    & \squash*{X} \ar[ur, "\overline{f}"', dashed]
  \end{tikzcd}
\]
Observe that the factorisation \(\overline{f}\) is unique by function
extensionality and the fact that \(P\) is a proposition.

Our paper can be summarised as follows:

\paragraph*{Section~\ref{sec:basicdomaintheory}.}
We introduce the theory of dcpos with \(\bot\) (known as domain theory) in
predicative constructive univalent type theory. We take the carriers of the
dcpos to be sets (in the sense of univalent type theory) and the partial orders
to be proposition-valued. Propositional truncation plays an import part in
defining directedness.

\paragraph*{Section~\ref{sec:constructiveissuespartiality}.}
We elaborate on the issue with the classical construction of the Scott model in
a constructive meta-theory (cf.\ the final paragraph of
Section~\ref{background:models}).

\paragraph*{Section~\ref{sec:lifting}.}
To remedy this issue, we work instead with the lifting monad (also known as the
partial map classifier monad) from topos theory~\cite{Kock1991}, which has
been extended to constructive type theory by Reus~and~Streicher~\cite{ReusStreicher1999} and
recently to univalent type theory by Escard\'o~and~Knapp~\cite{EscardoKnapp2017,Knapp2018}.  The
lifting $\lift(X)$ of a type $X$ is defined as
${\lift(X) \colonequiv \sum_{P : \Omega} (P \to X)}$, where $\Omega$ is the type
of propositions in the first universe. We think of the elements \((P,\phi)\) of
\(\lift(X)\) as partial elements of \(X\): in case \(P\) holds, we get an
element of \(X\), but \(P\) may also fail to hold and then the partial element
is thought of as undefined. In our constructive model, we interpret the base
type of PCF as the lifting \(\lift(\natt)\) of the natural numbers.

\paragraph*{Section~\ref{sec:PCF}.}
We define a combinatory version of PCF and its (small-step) operational
semantics. We use the propositional truncation to obtain well-behaved relations
in the small-step operational semantics.

\paragraph*{Section~\ref{sec:Scottmodel}.}
We define our constructive Scott model of PCF using the lifting monad.

\paragraph*{Section~\ref{sec:soundnesscompadequacy}.}
We show how the usual proofs of soundness and computational adequacy adapt to
our constructive setting with propositional truncations.

\paragraph*{Section~\ref{sec:charprop}.}
Recall that in our model the PCF type $\iota$ for natural numbers is interpreted
as $\lift(\natt)$, where $\natt$ is the natural numbers type. Thus, if $t$ is a
PCF term of type $\iota$, then we get an element $\densem{t} :
\lift(\natt)$. Hence, for every such term $t$ we have a proposition
$\fst(\densem{t}) : \Omega$. We show that such propositions are all
semidecidable. This result should be contrasted with the fact that a restricted
version of the lifting monad where we take a $\Sigma$-type over only
semidecidable propositions is not adequate for our purposes, as we explain at
the end of Section~\ref{sec:charprop}.

In proving our results, we take the opportunity to record some more general
properties of reflexive transitive closures (Section~\ref{semidec-rel}) and
indexed $\mathsf{W}$-types (Section~\ref{indexedWtypes}).

\paragraph*{Section~\ref{sizematters}.}
We discuss the universe levels involved in our development. This is important,
because we want our results to go through predicatively, i.e.\ without
propositional resizing.

\paragraph*{Section~\ref{conclusion}.}
We summarise our main results and describe directions for future work.

\subsection{Related work}\label{sec:relatedwork}
Partiality in type theory has been the subject of recent study. We briefly
discuss the different approaches.

Firstly, there are the delay monad by Capretta~\cite{Capretta2005} and its quotient
by weak bisimilarity, as studied by Chapman et al.~\cite{ChapmanUustaluVeltri2017}. They used
countable choice to prove that the quotient is again a monad.
Escard\'o and Knapp~\cite{EscardoKnapp2017,Knapp2018} showed that a weak form
of countable choice is indeed necessary to prove this. However, Coquand, Mannaa,
and Ruch \cite[Corollary~2]{CoquandMannaaRuch2017} have shown that countable
choice cannot be proved in dependent type theory with one univalent universe and
propositional truncation. Theorem~3.3 of Coquand's~\cite{Coquand2018} extends this to
dependent type theory with a hierarchy of univalent universes and (some) higher
inductive types. Moreover, Andrew Swan~\cite{Swan2019HoTTEST,Swan2019} recently
showed that even the weak form of choice required is not provable in univalent
type theory.

Another approach is laid out by Altenkirch, Danielsson and Kraus.~\cite{AltenkirchDanielssonKraus2017}. They
postulated the existence of a particular quotient inductive-inductive type
(QIIT) and showed that it satisfies the universal property of the free
$\omega$-cpo with a least element \cite[Theorem~5]{AltenkirchDanielssonKraus2017}. Moreover, Altenkirch et al.\ showed that,
assuming countable choice, their QITT coincides with the quotiented delay monad.

We stress that our approach does not need countable choice or quotient
inductive-inductive types.

Finally, Benton, Kennedy and Varming~\cite{BentonKennedyVarming2009} used
Capretta's delay monad to give a constructive approach to domain theory. Their
approach used setoids, so that every object comes with an equivalence relation
that maps must preserve. One cannot quotient these objects, because quotienting
Capretta's delay monad requires (a weak form of) countable choice, as explained
above. In our development, we instead use Martin-L\"of's identity types as our
notion of equality. Moreover, we do not make use of Coq's impredicative
\texttt{Prop} universe and our treatment incorporates directed complete posets
(dcpos) and not just $\omega$-cpos.

\subsection{Formalisation}
All our results up to and including the proof of computational adequacy (and
except for Section~\ref{sec:constructiveissuespartiality} and
Remark~\ref{liftasfree}) have been formalised in the proof assistant Coq using
the UniMath library \cite{UniMath} and Coq's \texttt{Inductive} types. The
general results from Section~\ref{sec:charprop} have also been formalised, but their
direct applications to PCF, e.g.\ single-valuedness of the operational semantics
and PCF as an indexed $\mathsf{W}$-type, have not. The code may be found at
\url{https://github.com/tomdjong/UniMath/tree/paper}. Instructions for use can
be found in the repository's
\href{https://github.com/tomdjong/UniMath/blob/paper/README.md}{\nolinkurl{README.md}}
file. Browsable documentation for the formalisation may be found at
\url{https://tomdjong.github.io/Scott-PCF-UniMath/toc.html}. Definitions and
proofs of lemmas, propositions and theorems are labelled with their
corresponding identifiers in the Coq name, for example as
\coqident{Partiality}{PCF}{pcf}, which also functions as a hyperlink to the
appropriate definition in the documentation.

At present, it is not possible to verify universe levels in UniMath. Therefore,
to verify the correctness of our development and our claims in
Section~\ref{sizematters} about universe levels in particular, we reformalised
part of our development in Agda using Mart\'in Escard\'o library~\cite{TypeTopology}. Our
code is now part of the library. An HTML rendering may be found at:
\url{https://www.cs.bham.ac.uk/~mhe/agda-new/PCFModules.html}.

\subsection{Acknowledgements}
Firstly, I would like to thank Mart\'in Escard\'o for suggesting and supervising
this project. Secondly, I am grateful to Benedikt Ahrens for his support, his
help with UniMath, and in particular for his feedback on earlier versions of
this paper. I should also like to thank Andrej Bauer and Bernhard Reus for their
comments and questions. Finally, I am indebted to the anonymous referees for
their thorough and valuable reports that helped to improve the paper.

\section{Basic domain theory}\label{sec:basicdomaintheory}
We introduce basic domain theory in the setting of constructive predicative
univalent mathematics. We adapt known definitions (cf.\
\cite[Section~2.1]{AbramskyJung1994} and \cite[Chapter~4]{Streicher2006}) to
constructive univalent type theory, paying special attention to how our
definitions may involve propositional truncations.

\subsection{Directed complete posets}
\begin{definition}[\coqident{Foundations}{Sets}{PartialOrder}]
  \label{def:poset}
  A \emph{poset} $(X,\leq)$ is a set $X$ together with a proposition-valued
  binary relation $\leq : X \to X \to \Prop$ satisfying:
  \begin{enumerate}
  \item \emph{reflexivity}: $\prod_{x : X} x \leq x$;
  \item \emph{antisymmetry}: $\prod_{x,y : X} x \leq y \to y \leq x \to x = y$;
  \item \emph{transitivity}:
    $\prod_{x,y,z : X} x \leq y \to y \leq z \to x \leq z$.
  \end{enumerate}
\end{definition}

\begin{remark}
  Notice that we require $\leq$ to take values in $\Prop$, the type of
  propositions in \(\mathcal{U}_0\), cf.\ Example~\ref{ex:posetuniv}. This
  allows us to prove (using function
  extensionality~\cite[Example~3.6.2]{HoTTbook}) that reflexivity and
  transitivity are propositions, i.e.\ there is at most one witness of
  reflexivity and transitivity. We also express this by saying that reflexivity
  and transitivity are properties, rather than structures. Moreover, we restrict
  to $X$ being a set to ensure that antisymmetry is a property, rather than a
  structure.
\end{remark}

\begin{definition}[\coqident{Foundations}{Sets}{posetmorphism}]
  Let $X$ and $Y$ be posets. A \emph{poset morphism} from $X$ to $Y$ is a
  function between the underlying sets that preserves the order. We also say
  that the function is \emph{monotone}.
\end{definition}

\begin{definition}[\coqident{Algebra}{DCPO}{isdirected}]
  \label{def:isdirected}
  Let $(X,\leq)$ be a poset and $I$ any type. Given a family $u : I \to X$, we
  often write $u_i$ for $u(i)$. Such a family is called \emph{directed} if \(I\)
  is inhabited (i.e.\ $\squash I$ holds) and
  \(
    \prod_{i,j : I} \squash*{\sum_{k : I}{\pa*{u_i \leq u_k}
        \times \pa*{u_j \leq u_k}}}
  \).
\end{definition}

\begin{remark}
  We use the propositional truncation in the definition above to ensure that
  being directed is a property, rather than a
  structure~(\coqident{Algebra}{DCPO}{isaprop_isdirected}).

  Firstly, we express that the type $I$ is inhabited by requiring an element of
  $\squash{I}$. This is different from requiring an element of $I$. It is akin
  to the difference (in set theory) between a set $X$ such that
  $\exists {x \in X}$ holds and a pair $(X,x)$ of a set with a chosen element
  $x \in X$.

  Secondly, if we had used an untruncated $\Sigma$ in the second clause of the
  definition, then we would have asked our poset to be equipped with an
  \emph{operation} mapping pairs $(x,y)$ of elements to some \emph{specified}
  element greater than both $x$ and $y$.
\end{remark}

\begin{definition}[\coqident{Algebra}{DCPO}{isupperbound},
  \coqident{Algebra}{DCPO}{islub},
  \coqident{Algebra}{DCPO}{isdirectedcomplete}]\label{def:directedcomplete}
  An element \(x\) of a poset \(X\) is an \emph{upper bound} of a family
  \(u : I \to X\) if \(u_i \sqsubseteq x\) for every \(i : I\).
  It is a \emph{least upper bound} of \(u\) if it is an upper
  bound and \(x \sqsubseteq y\) holds whenever \(y\) is an upper bound of \(u\).

  A poset \(X\) is called \emph{\(\mathcal U\)-directed complete} for a type
  universe \(\mathcal U\) if every directed family in \(X\) indexed by a type in
  \(\mathcal U\) has a least upper bound in \(X\), which we denote by
  \(\bigsqcup_{i : I} u_i\).
  Symbolically,
  \(
    \prod_{I : \mathcal U}\prod_{u : I \to X} \pa*{u\text{ is directed} \to
    \sum_{x : X} x \text{ is a least upper bound of u}}.
  \)

  We call such a poset a \emph{$\mathcal U$-dcpo}.
  We shall often simply write dcpo, omitting reference to the type universe.
\end{definition}

\begin{remark}
  Contrary to Definition~\ref{def:isdirected}, directed completeness is not
  phrased with a truncated~\(\Sigma\). This justifies having the least upper
  bound operator \(\bigsqcup\). The reason for this definition of directed
  completeness is that least upper bounds are unique when they
  exist~(\coqident{Algebra}{DCPO}{lubsareunique}). Moreover, the type expressing
  that an element is a least upper bound for a family can be shown to be a
  proposition using function
  extensionality~(\coqident{Algebra}{DCPO}{isaprop_islub}). Hence, for any
  family \(u\), the type of least upper bounds of \(u\) and its propositional
  truncation are equivalent.  This observation also tells us, using function
  extensionality again, that the type expressing that a poset is directed
  complete is also a
  proposition~(\coqident{Algebra}{DCPO}{isaprop_isdirectedcomplete}), i.e.\ it
  is a property of the poset.
\end{remark}

\begin{remark}
  In classical mathematics, a dcpo is usually defined as a poset such that every
  directed \emph{subset} has a least upper bound. We have formulated our version
  using \emph{families}, because in our type-theoretic framework functions are
  primitive, unlike in set theory where sets are primitive and functions are
  encoded as particular sets. Another reason for preferring families is that we
  work in the absence of propositional resizing, so that we must pay attention
  to size and therefore only ask for least upper bounds of \emph{small} directed
  subsets. This point is explained and worked out in detail in
  \cite[Section~5]{deJongEscardo2021b} to which we refer the interested
  reader. Here we limit ourselves to saying that working with families is more
  direct, and that for the Scott model we will only need to consider simple
  \(\natt\)-indexed directed families anyway.
\end{remark}

\subsection{Morphisms of dcpos}
\begin{definition}[\coqident{Algebra}{DCPO}{isdcpomorphism}]
  Let $D$ and $E$ be dcpos. A poset morphism from $D$ to $E$ is a \emph{dcpo
    morphism} (or \emph{continuous}) if it preserves least upper bounds of
  directed families. That is, if \(u : I \to D\) is a directed family, then
  \(f\pa*{\bigsqcup_{i : I}u_i}\) is the least upper bound of
  \(f \circ u : I \to E\).
\end{definition}
Thus, by definition, a dcpo morphism is required to be a poset morphism, i.e.\
it must be monotone. However, as is well-known in domain theory, requiring that
the function is monotone is actually redundant, as the following lemma shows.
\begin{lemma}\label{monotoneforfree}
  Let $D$ and $E$ be dcpos. If $f$ is a function \textup{(}on the underlying
  types\textup{)} from $D$ to $E$ preserving least upper bounds of directed
  families, then $f$ is order preserving.
\end{lemma}
\begin{proof}[\coqidentproof{Algebra}{DCPO}{preservesdirectedlub_isdcpomorphism}]
  Let $f : D \to E$ be a morphism of dcpos and suppose $x,y : D$ with
  $x \leq y$. Consider the family $\unitt + \unitt \to D$ defined as
  $\inl(\star) \mapsto x$ and $\inr(\star) \mapsto y$. This family is easily
  seen to be directed and its least upper bound is $y$. Now $f$ preserves this
  least upper bound, so $f(x) \leq f(y)$.
\end{proof}
\begin{lemma}
  Every morphism of dcpos preserves directed families. That is, if $f : D \to E$
  is a morphism of dcpos and $u$ is a directed family in $D$, then $f \circ u$
  is a directed family in $E$.
\end{lemma}
\begin{proof}[\coqidentproof{Algebra}{DCPO}{dcpomorphism_preservesdirected}]
  Using monotonicity of \(f\).
\end{proof}

\begin{therm}
  \label{dcpoofdcpomor}
  Let $D$ and $E$ be dcpos. The morphisms from $D$ to $E$ form a dcpo with the
  pointwise order.
\end{therm}
\begin{proof}[\coqidentproof{Algebra}{DCPO}{dcpoofdcpomorphisms}]
  The least upper bound of a directed family of dcpo morphisms is also given
  pointwise. The proof only differs from the standard proof of
  \cite[Theorem~4.2]{Streicher2006} in that it uses directed families, rather
  than subsets. One may consult the formalisation for the technical details.
\end{proof}

\subsection{Dcpos with \texorpdfstring{$\bot$}{bottom}}
\begin{definition}[\coqident{Algebra}{DCPO}{dcpowithbottom}]
  \label{def:dcpowithbottom}
  A \emph{dcpo with $\bot$} is a dcpo $D$ together with a least element in $D$.
\end{definition}

\begin{therm}
  \label{dcpowithbot_ofdcpomor}
  Let $D$ be a dcpo and let $E$ be a dcpo with $\bot$. Ordered pointwise, the
  morphisms from $D$ to $E$ form a dcpo with $\bot$, which we denote by $E^D$.
\end{therm}
\begin{proof}[\coqidentproof{Algebra}{DCPO}{dcpowithbottom_ofdcpomorphisms}]
  Since the order is pointwise, the least morphism from $D$ to $E$ is simply
  given by mapping every element in $D$ to the least element in $E$. The rest is
  as in Theorem~\ref{dcpoofdcpomor}.
\end{proof}

Dcpos with bottom elements are interesting because they admit least fixed
points. Moreover, these least fixed points are themselves given by a continuous
function.
\begin{therm}
  \label{leastfixedpoint}
  Let $D$ be a dcpo with $\bot$. There is a continuous function
  $\mu : D^D \to D$ that sends each continuous function to its least fixed
  point. In fact, $\mu$ satisfies:
  \begin{enumerate}
  \item $f(\mu(f)) = \mu(f)$ for every continuous $f : D \to D$;
  \item for every continuous $f : D \to D$ and each $d : D$, if $f(d) \leq d$,
    then $\mu(f) \leq d$.
  \end{enumerate}
\end{therm}
\begin{proof}
  [Proof. \textnormal{(\coqident{Algebra}{DCPO}{leastfixedpoint_isfixedpoint},
  \coqident{Algebra}{DCPO}{leastfixedpoint_isleast})}]
  We have formalised the proof of \cite[Theorem 2.1.19]{AbramskyJung1994}. We
  sketch the main construction here. For each natural number $n$, define
  $\mathsf{iter}(n) : D^D \to D$ as
  \[
    \mathsf{iter}(n)(f) \colonequiv f^n(\bot) \colonequiv
    \underbrace{f(f(\dots(f}_{n\text{ times}}(\bot))\dots)).
  \]
  By induction on $n$, one may show that every $\mathsf{iter}(n)$ is
  continuous. Then, the assignment $n~\mapsto~{\mathsf{iter}(n)}$ is a directed
  family in $D^{\pa*{D^D}}$. Finally, one defines $\mu$ as the least upper bound
  of this directed family. Recall that least upper bounds in the exponential are
  given pointwise, so that $\mu(f) = \bigsqcup_{n : \natt} f^n(\bot)$.
\end{proof}

\section{Constructive issues with partiality}
\label{sec:constructiveissuespartiality}
In classical mathematics, a partial map from $\mathbb N$ to $\mathbb N$ can
simply be seen as a total map from $\mathbb N$ to $\mathbb N \cup \{\bot\}$,
where $\bot$ is some fresh element not in $\mathbb N$. The \emph{flat dcpo
  $\mathbb N_\bot$} is $\mathbb N \cup \{\bot\}$ ordered as in the following
Hasse diagram:
\[
  \begin{tikzcd}
    0 \arrow[drr, no head] & 1 \ar[dr, no head] & 2 \ar[d, no head]
    & 3 \ar[dl, no head] & \hspace{-1em}\cdots \ar[dll, no head] \\
    & & \bot
  \end{tikzcd}
\]
Using excluded middle, a directed subset of \(\mathbb N_\bot\) is either
\(\{\bot\}\), \(\{n\}\) or \(\{\bot,n\}\) (with \(n\) a natural number). The
least upper bounds of which are easily computed as \(\bot\), \(n\) and \(n\),
respectively. Thus, with excluded middle, \(\mathbb N_\bot\) is directed
complete.

One could hope that the above translates directly into constructive univalent
mathematics, that is, that the poset
$\natt_{\bot} \colonequiv \pa*{\natt+\unitt,\leq_{\bot}}$ with $\leq_{\bot}$ the
flat order (i.e.\ $\inr(\star)$ is the least element and all other elements are
incomparable) is ($\mathcal{U}_0$-)directed complete (in the sense of
Definition~\ref{def:directedcomplete}). However, we can prove that this implies
Bishop's Limited Principle of Omniscience (LPO), a constructive taboo (recall
the final paragraph of Section~\ref{background:models}), as follows.

Write \(\twot\) for the type \(\unitt + \unitt\), and \(0\) and \(1\) for its
inhabitants \(\inl(\star)\) and \(\inr(\star)\), respectively. In~type~theory,
LPO may be formulated\footnote{This formulation does not ensure that the type is
  a proposition, so one could also consider truncating the $\Sigma$ or asking
  for the \emph{least} $k$ such that $\alpha(k) = 1$. But this version is
  sufficient for our purposes, and logically equivalent to the one with the
  truncated $\Sigma$.} as the following type:
\begin{equation}\label{LPO}
  \prod_{\alpha : \natt \to \twot} \pa*{\prod_{n : \natt} \alpha(n) = 0} +
  \pa*{\sum_{k : \natt} \alpha(k) = 1}.
  \tag{LPO}
\end{equation}

\begin{lemma}
  Directed completeness of $\natt_{\bot}$ implies \textup{LPO}.
\end{lemma}
\begin{proof}
  Suppose that $\natt_{\bot}$ is ($\mathcal{U}_0$-)directed complete. Let
  $\alpha : \natt \to \twot$ be an arbitrary binary sequence. Define the family
  $\beta : \natt \to \natt_{\bot}$ as
  \[
    \beta(n) \colonequiv
    \begin{cases}
      \inl(k) &\text{if $k$ is the least integer $\leq n$ such that $\alpha(k) =
        1$;} \\
      \inr(\star) &\text{else}.
    \end{cases}
  \]
  Then $\beta$ is directed, so by assumption, it has a supremum $s$ in
  $\natt_{\bot}$. By the induction principle of sum-types, we can decide whether
  $s = \inl(k)$ for some $k : \natt$ or $s = \inr(\star)$. The former implies
  $\sum_{k : \natt} \alpha(k) = 1$ and we claim that the latter implies
  $\prod_{n : \natt} \alpha(n) = 0$. For suppose that \(s = \inr(\star)\) and
  let \(n : \natt\). Since \(\twot\) has decidable equality, it suffices to show
  that \(\alpha(n) \neq 1\). Assume for a contradiction that \(\alpha(n) =
  1\). Then \(\beta(n) = \inl(k)\) for some natural number \(k \leq n\). Using
  that \(s\) is the supremum of \(\beta\) yields:
  \(\inl(k) = \beta(n) \leq_\bot s = \inr(\star)\). By definition of the order
  we also have the reverse inequality \(\inr(\star) \leq_\bot \inl(k)\). Hence,
  \(\inr(\star) = \inl(k)\) by antisymmetry, which is a contradiction, so
  \(\alpha(n) \neq 1\) as desired.
\end{proof}

\section{Partiality, constructively}\label{sec:lifting}
In this section we present the lifting monad as a solution to the problem
described in the previous section. Using the lifting monad in univalent type
theory to deal with partiality originates with the work of
Escard\'o~and~Knapp~\cite{EscardoKnapp2017,Knapp2018} and aims to avoid
countable choice.

We start by defining the lifting of a type and by characterising its identity
type. In Section~\ref{subsec:liftingmonad} we prove that the lifting carries a
monad structure, while in Section~\ref{subsec:liftingdcpo} we show that the
lifting of a set is a dcpo with $\bot$.  Most~of the definitions and some of the
results in this section can be found in~\cite{Knapp2018} or
in~\cite{EscardoKnapp2017}. Exceptions are Lemma~\ref{lifteqchar},
Theorem~\ref{liftofsetisaset} and Theorem~\ref{Kleisliextensioncont}. We note
that our characterisation of equality of the lifting, Lemma~\ref{lifteqchar}, is
implicit in the fact that the order of~\cite{EscardoKnapp2017} is
antisymmetric. The order on the lifting in this paper (see
Theorem~\ref{liftofsetisdcpo}) is different from the order presented in
\cite{EscardoKnapp2017,Knapp2018}. The two orders are equivalent, however, as
observed by in~\cite[\texttt{LiftingUnivalentPrecategory}]{TypeTopology}. We
found the order in this paper to be more convenient.

\begin{definition}[\coqident{Partiality}{PartialElements}{lift}]
  Let $X$ be any type. Define the \emph{lifting of X} as
  \[
    \lift(X) \colonequiv \sum_{P : \Prop} (P \to X).
  \]
  Strictly speaking, we should have written $\fst(P) \to X$, because
  elements of $\Omega$ are pairs of types and witnesses that these types are
  subsingletons. We will almost always suppress reference to these witnesses in
  this paper.
\end{definition}

\begin{definition}[\coqident{Partiality}{PartialElements}{liftorder_least}]
  \label{lift_least}
  For any type $X$, the type $\lift(X)$ has a distinguished element
  \[
    \bot_X \colonequiv \pa*{\emptyt,\textup{\textsf{from-}}\emptyt_X} :
    \lift(X),
  \]
  where $\textup{\textsf{from-}}\emptyt_X$ is the unique function from $\emptyt$
  to $X$.
\end{definition}

\begin{definition}[\coqident{Partiality}{PartialElements}{lift_embedding}]
  \label{lift_embedding}
  There is a canonical map $ \eta_X \colon X \to \lift(X) $ defined by
  \[
    \eta_X(x) \colonequiv (\unitt,\lambda t.x).
  \]
\end{definition}
Assuming LEM (i.e.\ $\prod_{P : \Omega} (P + \lnot P)$), we can prove that the
only propositions are \(\emptyt\) and \(\unitt\), for if a proposition \(P\)
holds, then it is equal (by propositional extensionality) to \(\unitt\) and if
it does not hold, then it is equal to \(\emptyt\). Hence, if we
assume LEM then the two definitions above capture all of the lifting, since LEM
implies:
\[
  \lift(X) \equiv \pa*{\sum_{P : \Omega}(P \to X)}
  \simeq \pa*{\pa*{\unitt \to X} + \pa*{\emptyt \to X}}
  \simeq \pa*{X + \unitt},
\]
as \(\pa*{\unitt \to X} \simeq X\) and there is a unique function from
\(\emptyt\) to any type \(X\).
Constructively, things are more interesting, of course.

We proceed by defining meaningful projections.
\begin{definition}
  [\coqident{Partiality}{PartialElements}{isdefined},
  \coqident{Partiality}{PartialElements}{value}] We take
  $\isdefined : \lift(X) \to \Prop$ to be the first projection. The function
  $\liftvalue : \prod_{l : \lift (X)} \isdefined (l) \to X$ is given by:
  $\liftvalue (P,\phi) (p) \colonequiv \phi (p)$.
\end{definition}

Since equality of $\Sigma$-types often requires $\transport$, it will be
convenient to characterise the equality of $\lift(X)$.
\begin{lemma}
  \label{lifteqchar}
  Let $X$ be any type and let $l, m : \lift(X)$. The following are logically
  equivalent\footnote{In fact, there is a type equivalence. One can prove this
    using univalence and a generalised structure identity principle, cf.\
    \cite[\texttt{LiftingIdentityViaSIP}]{TypeTopology}.}
  \begin{enumerate}
  \item $l = m$\textup{;}
  \item\label{lifteqchar-2}
    $\sum_{e : \isdefined(l) \leftrightarrow \isdefined(m)} \liftvalue(l) \circ
    \snd(e) \sim \liftvalue(m)$.
  \end{enumerate}
\end{lemma}
\begin{proof}
  [Proof. \textnormal{(\coqident{Partiality}{PartialElements}{lifteq_necc},
    \coqident{Partiality}{PartialElements}{lifteq_suff})}]

  First of all, the characterisation of the identity type of
  \(\Sigma\)-types~\cite[Theorem~2.7.2]{HoTTbook} yields:
  \begin{equation*}\label{id-type-of-sum-type}
    \tag{\(\dagger\)}
    \pa*{l = m} \simeq \sum_{e' : \isdefined(l) = \isdefined(m)}
    \transport(e',\liftvalue(l))=\liftvalue(m).
  \end{equation*}
  Thus we only have to show that the right-hand side of
  \eqref{id-type-of-sum-type} is logically equivalent to \eqref{lifteqchar-2} in the lemma.
  Suppose first that we have \(e' : \isdefined(l) = \isdefined(m)\) and an
  equality \break \({p : \transport(e',\liftvalue(l)) = \liftvalue(m)}\). Then
  \[
    e \colonequiv \mathsf{eqtoiff}(e') : {\isdefined(l) \leftrightarrow
      \isdefined(m)}.
  \]
  Using path induction on \(e'\), we can prove that
  $\liftvalue(l) \circ \snd(e) =
  \transport(e',\liftvalue(l))$. Together with $p$, this equality
  implies
  $\liftvalue(l) \circ \snd(e) \sim \liftvalue(m)$,
  as desired.

  Conversely, suppose $e : \isdefined (l) \leftrightarrow \isdefined (m)$ and
  $v : \liftvalue (l) \circ \snd(e) \sim \liftvalue(m)$. By
  propositional extensionality, we obtain $e' : \isdefined(l) = \isdefined(m)$
  from $e$. From $e'$ we can get an equivalence
  $\mathsf{idtoeqv}(e') : \isdefined (l) \simeq \isdefined(m)$.  Furthermore,
  using path induction on $e'$, one can prove that
  \begin{equation*}\label{transporteq}
    \transport(e',\liftvalue(l)) = \liftvalue(l) \circ
    (\mathsf{idtoeqv}(e'))^{-1}.\tag{$\ast$}
  \end{equation*}
  Hence, it suffices to show that the right-hand side of~\eqref{transporteq} is
  equal to \(\liftvalue(m)\). The homotopy \(v\) yields
  \(\liftvalue(l) \circ \snd(e) = \liftvalue(m)\) by function
  extensionality, so it suffices to prove that
  \((\mathsf{idtoeqv}(e'))^{-1} = \snd(e)\). But these are both
  functions with codomain \(\isdefined(l)\), which is a proposition, so they are
  equal by function extensionality.
\end{proof}

\subsection{The lifting monad}\label{subsec:liftingmonad}
In this section we prove that the lifting carries a monad structure.

This monad structure is most easily described as a Kleisli triple. The unit is
given by Definition~\ref{lift_embedding}.
\begin{definition}[\coqident{Partiality}{LiftMonad}{Kleisli_extension}]
  Given $f : X \to \lift (Y)$, the \emph{Kleisli extension}
  $f^\# : \lift(X) \to \lift(Y)$ is defined by:
  \[
    f^\#(P,\phi) \colonequiv \Bigg(\sum_{p : P} \isdefined (f(\phi(p))),
    \psi\Bigg),
  \]
  where $\psi(p,d) \colonequiv \liftvalue (f(\phi(p)))(d)$.
\end{definition}

\begin{therm}[Theorem 5.8 in \cite{Knapp2018}, Section 2.2 in \cite{EscardoKnapp2017}]
  \label{lifting-is-monad}
  The above constructions yield a monad structure on $\lift(X)$, i.e.\ the
  Kleisli laws hold \textup{(}pointwise\textup{)}:
  \begin{enumerate}
  \item $\pa*{\eta_X}^\# \sim \id_{\lift(X)}$;
  \item $f^\# \circ \eta_X \sim f$ for any $f : X \to \lift (Y)$;
  \item\label{lifting-is-monad-3}
    $g^\# \circ f^\# \sim (g^\# \circ f)^\#$ for any $f : X \to \lift(Y)$
    and $g : Y \to \lift(Z)$.
  \end{enumerate}
\end{therm}
\begin{proof}
  [Proof. \textnormal{(\coqident{Partiality}{LiftMonad}{eta_extension},
  \coqident{Partiality}{LiftMonad}{fun_extension_after_eta},
  \coqident{Partiality}{LiftMonad}{extension_comp})}]
  The proofs are straightforward thanks to Lemma~\ref{lifteqchar}.
  Item~\eqref{lifting-is-monad-3} is essentially the
  associativity of $\Sigma$, i.e.\ equivalence between
  $\sum_{a : A} \sum_{b : B(a)} C(a,b)$ and
  $\sum_{(a,b) : \sum_{a : A} B(a)} C(a,b)$.
\end{proof}

\subsection{The lifting as a dcpo with \texorpdfstring{$\bot$}{bottom}}
\label{subsec:liftingdcpo}
The goal of this section is to endow $\lift(X)$ with a partial order that makes
it into a dcpo~with~$\bot$, provided that $X$ is a set. We also show that the
Kleisli extension from the previous section is continuous when regarded as a
morphism between dcpos~with~$\bot$.

\begin{therm}
  \label{liftofsetisaset}
  If $X$ is a set, then so is its lifting $\lift(X)$.
\end{therm}
\begin{proof}[\coqidentproof{Partiality}{PartialElements}{liftofhset_isaset}]
  As in the proof of Lemma~\ref{lifteqchar}, we have:
  \[
    l = m \simeq \sum_{e : \isdefined(l) = \isdefined(m)}
    \transport(e,\liftvalue (l)) = \liftvalue (m).
  \]
  Since $X$ is a set, the type
  $\transport(e,\liftvalue(l)) = \liftvalue(m)$ is a proposition. So, if
  we can prove that $\isdefined(l) = \isdefined(m)$ is a proposition, then the
  right hand side is a proposition indexed sum of propositions, which is again a
  proposition.

  So let us prove that if $P$ and $Q$ are propositions, then so is $P = Q$. At
  first glance, it might seem like one needs univalence (for propositions) to
  prove this, but in fact propositional extensionality suffices. By
  \cite[Lemma~3.11]{Krausetal2017} (applied to the type of propositions), it
  suffices to give for every proposition \(R\), a (weakly) constant (i.e.\ any
  two of its values are equal) endomap on $P = R$. But the composition
  \[
    (P = R) \to (P \leftrightarrow R) \xrightarrow{\text{PropExt}} (P = R)
  \]
  is weakly constant, because $P \leftrightarrow R$ is a proposition, so this
  finishes the proof.
\end{proof}

\begin{therm}[cf.\ Theorem 5.14 in \cite{Knapp2018} and Theorem 1 in
  \cite{EscardoKnapp2017}]
  \label{liftofsetisdcpo}
  If $X$ is a set, then $\lift(X)$ is a dcpo with $\bot$ with the following
  order:
  \[
    l \sqsubseteq m \colonequiv \isdefined(l) \to l = m.
  \]
\end{therm}
\begin{proof}
  [\coqidentproof{Partiality}{PartialElements}{liftdcpowithbottom}] First of
  all, we should prove that $\lift(X)$ is a poset with the specified order. In
  particular, $\sqsubseteq$ should be proposition-valued. If $X$ is a set, then
  $\isdefined(l) \to l = m$ is a function type into a proposition and therefore
  a proposition itself.

  Reflexivity and transitivity of $\sqsubseteq$ are easily verified. Moreover,
  $\sqsubseteq$ is seen to be antisymmetric using Lemma~\ref{lifteqchar}.

  The $\bot$ element of $\lift(X)$ is given by $\bot_X$ from
  Definition~\ref{lift_least}.

  The construction of the least upper bound of a directed family is the most
  challenging part of the proof. Let \(u : I \to \lift(X)\) be a directed family
  in \(\lift(X)\). Consider the diagram (of solid arrows):

  \begin{center}
    \begin{tikzcd}
      \sum_{i : I}\isdefined(u_i)
      \ar[rr,"\phi : {(i,d)} \mapsto \liftvalue(u_i)(d)"]
      \ar[dr,"\tosquash{-}"']
      & & X \\
      & \squash*{\sum_{i : I}\isdefined(u_i)}
      \ar[ur, dashed, "\psi"']
    \end{tikzcd}
  \end{center}
  We are going to construct the dashed map \(\psi\) that makes the diagram
  commute and define the least upper bound of \(u\) as:
  \(\pa*{\squash*{\sum_{i : I}\isdefined(u_i)},\psi}\). Truncating the type is
  necessary, as \(\sum_{i : I}\isdefined(u_i)\) may have more than one element
  if \(I\) is not a proposition. The difficulty lies in the fact that the
  universal property of the truncation only tells us how to define maps into
  \emph{propositions}. But \(X\) is a \emph{set}. We solve this problem using
  \cite[Theorem~5.4]{Krausetal2017}, which says that every weakly constant
  function \(f : A \to B\) to a set \(B\) factors through \(\squash*{A}\). That
  \(f\) is weakly constant means that \(f(a) = f(a')\) for every \(a,a' :
  A\). So, to construct \(\psi\), we only need to prove that the top map
  \(\phi\) in the diagram is weakly constant. Let \((i,d_i) , (j,d_j)\) be two
  elements of the domain of \(\phi\). We are to prove that
  \(\liftvalue(u_i)(d_i) = \liftvalue(u_j)(d_j)\). As \(X\) is a set, this is a
  proposition. Therefore, using that \(u\) is directed, we obtain \(k : I\) with
  \(u_i,u_j \sqsubseteq u_k\). But \(d_i : \isdefined(u_i)\) and
  \(d_j : \isdefined(u_j)\), so \(u_i = u_k = u_j\) by definition of the
  order. Hence,
  \(\phi(i,d_i) = \liftvalue(u_i)(d_i) = \liftvalue(u_j)(d_j) = \phi(j,d_j)\),
  as we wished to show.
\end{proof}

\begin{therm}
  \label{Kleisliextensioncont}
  Let $X$ and $Y$ be sets and $f : X \to \lift(Y)$ any function. The Kleisli
  extension $f^\# : \lift(X) \to \lift(Y)$ is a morphism of dcpos.
\end{therm}
\begin{proof}[\coqidentproof{Partiality}{LiftMonad}{Kleisli_extension_dcpo}]
  Let $v$ be the least upper bound of a directed family $u : I \to \lift(X)$ in
  $\lift(X)$. Proving that $f^\#$ is monotone is quite easy. By monotonicity,
  $f^\#(v)$ is an upper bound for the family $f^\# \circ u$. We are left to
  prove that it is the least. Suppose that $l : \lift(Y)$ is another upper bound
  for the family $f^\#\circ u$, i.e.\ $l \sqsupseteq f^\#(u_i)$ for every
  $i : I$. We must show that $f^\#(v) \sqsubseteq l$. To this end, assume we
  have $q : \isdefined(f^\#(v))$. We must prove that $f^\#(v) = l$.

  From $q$, we obtain $p : \isdefined(v)$ by definition of $f^\#$. By our
  construction of suprema in $\lift(X)$ and the fact that $f^\#(v) = l$ is a
  proposition, we may in fact assume that we have an element $i : I$ and
  $d_i : \isdefined(u_i)$. But $l \sqsupseteq f^\#(u_i)$, so using $d_i$, we get
  the equality $l = f^\#(u_i)$. Since $v$ is an upper bound for $u$, the term
  $d_i$ also yields $u_i = v$. In particular, $l = f^\#(u_i) = f^\#(v)$, as
  desired.
\end{proof}

\begin{remark}[\coqident{Partiality}{LiftMonad}{liftfunctor_eq}]
  \label{liftfunctoreq}
  Finally, one could define the functor $\lift$ from the Kleisli extension and
  unit by putting $\lift(f) \colonequiv (\eta_Y \circ f)^\#$ for any
  $f : X \to Y$. However, it is equivalent and easier to directly define
  $\lift(f)$ by postcomposition:
  $\lift(f)(P,\phi) \colonequiv (P,f \circ \phi)$.
\end{remark}

\begin{remark}\label{liftasfree}
  We remark that lifting may be regarded as a free construction, in more than
  one way in fact. This result should be compared to
  \cite[Theorem~5]{AltenkirchDanielssonKraus2017}, where Altenkirch et al.\
  exhibit their QIIT as the free $\omega$-cpo with a least element (cf.\
  Section~\ref{sec:relatedwork}).

  By~\cite[Theorems 21 and 23]{deJongEscardo2021a}, the lifting of a set \(X\)
  can be regarded both as the free pointed dcpo on \(X\) and as the free
  subsingleton complete poset on \(X\). In our predicative setting, some care
  should be taken in formulating these statements. We do not go into the details
  here and instead refer the interested reader to~\cite{deJongEscardo2021a}.
\end{remark}

\section{PCF and its operational semantics}\label{sec:PCF}
This section formally defines the types and terms of PCF as well as the
small-step operational semantics. It should be regarded as a formal counterpart
to the informal introduction to PCF in Section~\ref{background:PCF}.

To avoid dealing with free and bound variables (in the formalisation), we opt to
work in the combinatory version of PCF, as originally presented
by Scott~\cite{Scott1993}. We note that it is possible to represent every closed
$\lambda$-term in terms of combinators by a well-known
technique~\cite[Section~2C]{HindleySeldin2008}.

We inductively define combinatory PCF as follows.
\begin{definition}[\coqident{Partiality}{PCF}{type}]
  The \emph{PCF types} are inductively defined as:
  \begin{enumerate}
  \item $\iota$ is a type, the \emph{base type};
  \item for every two types $\sigma$ and $\tau$, there is a \emph{function type}
    $\sigma \Rightarrow \tau$.
  \end{enumerate}
  As usual, $\Rightarrow$ will be right associative, so we write
  $\sigma \Rightarrow \tau \Rightarrow \rho$ for
  $\sigma \Rightarrow (\tau \Rightarrow \rho)$.
\end{definition}
\begin{definition}[\coqident{Partiality}{PCF}{term}]
  \label{def:PCFterms}
  The \emph{PCF terms of PCF type $\sigma$} are inductively generated by:
  \begin{center}
    \AxiomC{\phantom{$t : \iota$}}
    \UnaryInfC{${\zeroo} \text{ of type } \iota$}
    \DisplayProof \quad
    \AxiomC{\phantom{$t : \iota$}}
    \UnaryInfC{${\succc} \text{ of type } \iota \Rightarrow \iota$}
    \DisplayProof\vspace{0.5cm}\\
    \AxiomC{\phantom{$t : \iota$}}
    \UnaryInfC{$\pred \text{ of type } \iota \Rightarrow \iota$}
    \DisplayProof \quad
    \AxiomC{\phantom{$t : \iota$}}
    \UnaryInfC{${\ifz} \text{ of type } {\iota \Rightarrow \iota \Rightarrow \iota
        \Rightarrow \iota}$}
    \DisplayProof \vspace{0.5cm}\\
    \AxiomC{\phantom{$t : \iota$}}
    \UnaryInfC{${\PCFk_{\sigma,\tau}} \text{ of type } \sigma \Rightarrow \tau \Rightarrow
      \sigma$}
    \DisplayProof \quad \AxiomC{\phantom{$t : \iota$}}
    \UnaryInfC{${\PCFs_{\sigma,\tau,\rho}} \text{ of type } (\sigma \Rightarrow \tau
      \Rightarrow \rho) \Rightarrow \break (\sigma \Rightarrow \tau) \Rightarrow
      \sigma \Rightarrow \rho$}
    \DisplayProof\vspace{0.5cm}\\
    \AxiomC{\phantom{$t : \iota$}}
    \UnaryInfC{${\fix_\sigma} \text{ of type } (\sigma \Rightarrow \sigma)
      \Rightarrow \sigma$}
    \DisplayProof \quad
    \AxiomC{$s \text{ of type } \sigma \Rightarrow \tau$}
    \AxiomC{$t \text{ of type } \tau$}
    \BinaryInfC{$(st) \text{ of type } \tau$}
    \DisplayProof
  \end{center}
  We will often drop the parentheses in the final clause, as well as the PCF
  type subscripts in \({\PCFk_{\sigma,\tau}}\), \({\PCFs_{\sigma,\tau,\rho}}\)
  and \({\fix_\sigma}\). Finally, we employ the convention that the parentheses
  associate to the left, i.e.\ we write $rst$~for~$(rs)t$.
\end{definition}
\begin{definition}[\coqident{Partiality}{PCF}{numeral}]
  \label{def:PCFnumerals}
  For any $n : \natt$, let us write $\underline n$ for the $n$th PCF numeral,
  defined inductively as:
  \[
    \underline 0 \colonequiv {\zeroo};\quad \underline {n+1} \colonequiv
    \succc\,\underline n.
  \]
\end{definition}

To define the small-step operational semantics of PCF, we first define the
following inductive type.
\begin{definition}[\coqident{Partiality}{PCF}{smallstep'},
  \coqident{Partiality}{PCF}{smallstep}]\label{def:small-step}
  Define the \emph{small-step pre-relation} $\smallsteppre$ of type
  \[
    \prod_{\sigma : \text{PCF types}} \text{PCF terms of type $\sigma$} \to
    \text{PCF terms of type $\sigma$} \to \mathcal{U}_0
  \]
  as the inductive family generated by:
  \begin{center}
    \AxiomC{\phantom{$f \smallsteppre g$}}
    \UnaryInfC{$\pred\underline 0\smallsteppre\underline 0$}
    \DisplayProof \quad \AxiomC{\phantom{$f \smallsteppre g$}}
    \UnaryInfC{$\pred\underline {n+1}\smallsteppre\underline n$}
    \DisplayProof \quad \AxiomC{\phantom{$f \smallsteppre g$}}
    \UnaryInfC{$\ifz s\,t\,\underline 0\smallsteppre s$} \DisplayProof
    \quad \AxiomC{\phantom{$f \smallsteppre g$}}
    \UnaryInfC{$\ifz s\,t\,\underline {n+1}\smallsteppre t$}
    \DisplayProof \vspace{0.5cm}
    \\
    \AxiomC{\phantom{$f \smallsteppre g$}}
    \UnaryInfC{${\PCFk} st\smallsteppre s$} \DisplayProof \quad
    \AxiomC{\phantom{$f \smallsteppre g$}}
    \UnaryInfC{${\PCFs}fgt\smallsteppre ft(gt)$} \DisplayProof \quad
    \AxiomC{\phantom{$f\smallsteppre g$}}
    \UnaryInfC{$\fix f\smallsteppre f(\fix f)$} \DisplayProof
    \quad \AxiomC{$f\smallsteppre g$} \UnaryInfC{$ft\smallsteppre gt$}
    \DisplayProof \vspace{0.5cm}
    \\
    \AxiomC{$s \smallsteppre t$}
    \UnaryInfC{$\succc s \smallsteppre \succc t$} \DisplayProof
    \quad \AxiomC{$s \smallsteppre t$}
    \UnaryInfC{$\pred s \smallsteppre \pred t$} \DisplayProof
    \quad \AxiomC{$r \smallsteppre r'$}
    \UnaryInfC{$\ifz s\,t\,r \smallsteppre \ifz s\,t\,r'$}
    \DisplayProof
  \end{center}

  We have been unable to prove that $s \smallsteppre t$ is a proposition for
  every suitable PCF terms $s$~and~$t$. The difficulty is that one cannot
  perform induction on \emph{both} \(s\) and \(t\). However, conceptually,
  $s \smallsteppre t$ should be a proposition, as (by inspection of the
  definition), there is at most one way by which we obtained
  $s \smallsteppre t$. Moreover, for technical reasons that will become apparent
  later, we really want $\smallsteppre$ to be propostion-valued.

  We solve the problem by defining the \emph{small-step relation} $\smallstep$ as
  the propositional truncation of~$\smallsteppre$, i.e.\
  $s \smallstep t \colonequiv \squash{s \smallsteppre t}$.
\end{definition}

\begin{remark}
  Benedikt Ahrens pointed out that in an impredicative framework, one could use
  propositional resizing and an impredicative encoding, i.e.\ by defining
  $\smallstep$ as a $\Pi$-type of all suitable proposition-valued
  relations. This is similar to the situation in set theory, where one would
  define $\smallstep$ as an intersection. Specifically, say that a relation
  \[
    R : \prod_{\sigma : \text{PCF types}} \pa*{\text{PCF terms of type
        \(\sigma\)} \to \text{PCF terms of type \(\sigma\)} \to \Omega_{\mathcal
        U_0}}
  \] is \emph{suitable} if it closed under all the clauses of
  Definition~\ref{def:small-step}, i.e.\
  \(R\pa*{\iota,\pred \underline 0,\underline 0}\),
  \(R\pa*{\iota,\pred \underline{n+1},\underline n}\), etc.\ are all
  inhabited.  We could define
  \(s \smallstep_{\text{impred}} t \colonequiv \prod_{R \text{ suitable}}
  R(\sigma,s,t)\). But notice the increase in universe level:
  \[
    \smallstep_{\text{impred}} : \prod_{\sigma : \text{PCF types}}
    \pa*{\text{PCF terms of type \(\sigma\)} \to \text{PCF terms of type
        \(\sigma\)} \to \Omega_{\mathcal U_1}}.
  \]
  So because of this increase, \(\smallstep_{\text{impred}}\) itself is not one
  of the suitable relations. Therefore \(\smallstep_{\text{impred}}\) does not
  satisfy the appropriate universal property in being the least relation closed
  under the clauses in Definition~\ref{def:small-step}. With propositional
  resizing we could resize \(\smallstep_{\text{impred}}\) to a
  \(\mathcal U_0\)-valued relation satisfying the appropriate universal
  property.
  The advantage of using the propositional truncation above is that it does
  satisfy the right universal property even without propositional resizing.
\end{remark}

  Let $R : X \to X \to \Prop$ be a relation on a type $X$. We might try to
  define the reflexive transitive closure $R_\ast$ of $R$ as an inductive type,
  generated by three constructors:
  \begin{alignat*}{2}
    &\mathsf{extend} &&: \prod_{x,y : X} x R y \to x R_\ast y; \\
    &\mathsf{refl}   &&: \prod_{x : X} x R_\ast x; \\
    &\mathsf{trans} &&: \prod_{x,y,z : X} x R_\ast y \to y R_\ast z \to x R_\ast
    z.
  \end{alignat*}
  But $R_\ast$ is not necessarily proposition-valued, even though $R$ is. This
  is because we might add a pair $(x,y)$ to $R_\ast$ in more than one way, for
  example, once by an instance of $\mathsf{extend}$ and once by an instance of
  $\mathsf{trans}$. Thus, we are led to the following definition.

\begin{definition}
  [\coqident{MoreFoundations}{ClosureOfHrel}{refl_trans_clos},
  \coqident{MoreFoundations}{ClosureOfHrel}{refl_trans_clos_hrel}]
  \label{def:refltransclos}
  Let $R : X \to X \to \Prop$ be a relation on a type $X$. We define the
  \emph{reflexive transitive closure} $R^\ast$ of $R$ by
  $x R^\ast y \colonequiv \squash*{x R_\ast y}$, where $R_\ast$ is as above.
\end{definition}
It is not hard to show that $R^\ast$ is the least reflexive and transitive
proposition-valued relation that extends $R$, so $R^\ast$ satisfies the
appropriate universal property
(\coqident{MoreFoundations}{ClosureOfHrel}{refl_trans_clos_univprop}).

Some properties of $\smallstep$ reflect onto $\smallstep^\ast$ as the following
lemma shows.
\begin{lemma}
  \label{refltransreduce}
  Let $r', r, s$ and $t$ be PCF terms of type $\iota$. If
  $r' \smallstep^\ast r$, then
  \begin{enumerate}
  \item\label{refltransreduce-1} $\succc r' \smallstep^\ast \succc r$;
  \item $\pred r' \smallstep^\ast \pred r$;
  \item $\ifz s\,t\,r' \smallstep^\ast \ifz s\,t\,r$.
  \end{enumerate}
  Moreover, if $f$ and $g$ are PCF terms of type $\sigma\Rightarrow\tau$ and
  $f \smallstep^\ast g$, then $f t \smallstep^\ast g t$ for any PCF term $t$ of
  type $\sigma$.
\end{lemma}
\begin{proof}
  [Proof \textnormal{(\coqident{Partiality}{PCF}{succ_refltrans_smallstep},
    \coqident{Partiality}{PCF}{pred_refltrans_smallstep},
    \coqident{Partiality}{PCF}{ifz_refltrans_smallstep},
    \coqident{Partiality}{PCF}{app_refltrans_smallstep})}] We only
  prove~\eqref{refltransreduce-1} the rest is similar. Suppose
  $r' \smallstep^\ast r'$.  Since $\succc r' \smallstep^\ast \succc r$ is a
  proposition, we may assume that we actually have a term $p$ of type
  $r' \smallstep_\ast r'$. Now we can perform induction on $p$. The cases were
  $p$ is formed using $\mathsf{refl}$ or $\mathsf{trans}$ are easy. If $p$ is
  formed by $\mathsf{extend}$, then we get a term of type
  $r \smallstep r' \equiv \squash{r \widetilde{\smallstep} r'}$. Again, as we
  are proving a proposition, we may suppose the existence of a term of type
  $r \widetilde{\smallstep} r'$. By Definition~\ref{def:small-step}, we
  then get $\succc r'\widetilde{\smallstep}\succc r$. This in turn yields,
  $\succc r'\smallstep\succc r$ and finally we use $\mathsf{extend}$ to get the
  desired $\succc r'\smallstep^\ast\succc r$.
\end{proof}

\section{The Scott model of PCF using the lifting monad}
\label{sec:Scottmodel}
Next, we wish to give a denotational semantics for PCF, namely the Scott model,
as explained in Definition~\ref{background:models}. We recall that the idea is
to assign some mathematical structure to each PCF type. The PCF terms are then
interpreted as elements of the structure.

\begin{definition}[\coqident{Partiality}{PCF}{denotational_semantics_type}]
  Inductively assign to each PCF type $\sigma$ a dcpo with $\bot$ as follows:
  \begin{enumerate}
  \item $\densem{\iota} \colonequiv \lift(\natt)$;
  \item
    $\densem{\sigma \Rightarrow \tau} \colonequiv
    \densem{\tau}^{\densem{\sigma}}$.
  \end{enumerate}
  Recall that if $D$ and $E$ are dcpos with $\bot$, then $E^D$ is the dcpo with
  $\bot$ of dcpo morphisms from $D$ to $E$, with pointwise ordering and
  pointwise least upper bounds.
\end{definition}

Next, we interpret PCF terms as elements of these dcpos with $\bot$, for which
we will need that \(\lift\) is a monad (with unit \(\eta\)) and (in particular)
a functor (recall Theorem~\ref{lifting-is-monad} and Remark~\ref{liftfunctoreq}).

\begin{definition}[\coqident{Partiality}{PCF}{denotational_semantics_terms}]
  Define for each PCF term $t$ of PCF type $\sigma$ a term $\densem{t}$ of type
  $\densem{\sigma}$, by the following inductive clauses:
  \begin{enumerate}
  \item $\densem{\zeroo} \colonequiv \eta(0)$;
  \item $\densem{\succc} \colonequiv \lift(s)$, where
    $s : \natt \to \natt$ is the successor function;
  \item $\densem{\pred} \colonequiv \lift(p)$, where
    $p : \natt \to \natt$ is the predecessor function;
  \item
    $\densem{\ifz} :
    \densem{\iota\Rightarrow\iota\Rightarrow\iota\Rightarrow\iota}$ is defined
    using the Kleisli extension as: $\lambda x,y.\pa*{\chi_{x,y}}^{\#}$, where
    \[
      \chi_{x,y}(n) \colonequiv
      \begin{cases}
        x &\text{if } n = 0; \\
        y &\text{else};
      \end{cases}
    \]
  \item $\densem{\PCFk} \colonequiv \lambda x, y . x$;
  \item $\densem{\PCFs} \colonequiv \lambda f, g, x . (f (x)) (g (x))$;
  \item $\densem{\fix} \colonequiv \mu$, where $\mu$ is the least fixed
    point operator from Theorem~\ref{leastfixedpoint}.
  \end{enumerate}
\end{definition}

\begin{remark}
  Of course, there are some things to be proved here. Namely,
  $\densem{\succc}, \densem{\pred}, \dots, \densem{\fix}$ all need to be dcpo
  morphisms. In the case of $\densem{\succc}$ and $\densem{\pred}$, we simply
  appeal to Theorem~\ref{Kleisliextensioncont} and
  Remark~\ref{liftfunctoreq}. For $\densem{\fix}$, this is
  Theorem~\ref{leastfixedpoint}. The continuity of
  $\densem{\PCFk}, \densem{\PCFs}$ and $\densem{\ifz}$ can be verified directly,
  as done in the formalisation (\coqident{Partiality}{PCF}{k_dcpo},
  \coqident{Partiality}{PCF}{s_dcpo},
  \coqident{Partiality}{PCF}{lifted_ifz}). It is however, unenlightning and
  tedious, so we omit the details here.
\end{remark}

As a first result about our denotational semantics, we show that the PCF
numerals have a canonical interpretation in the denotational semantics.
\begin{proposition}
  \label{densemnumeral}
  For every natural number $n$, we have $\densem{\underline n} = \eta (n)$.
\end{proposition}
\begin{proof}[\coqidentproof{Partiality}{PCF}{denotational_semantics_numerals}]
  We proceed by induction on $n$. The $n\equiv 0$ case is by definition of
  $\densem{\underline 0}$. Suppose $\densem{\underline m} = \eta (m)$ for a
  natural number $m$. Then,
  \begin{alignat*}{2}
    \densem{\underline {m+1}} &= \densem{\succc}(\densem{\underline
                                m}) \\
                              &= \lift(s)(\eta (m))\quad&&\text{(by induction
                                                          hypothesis)} \\
                              &= \eta(m + 1) \quad&&\text{(by definition of
                                                    the lift functor)},
  \end{alignat*}
  as desired.
\end{proof}

\section{Soundness and computational adequacy}
\label{sec:soundnesscompadequacy}
In this section we show that the denotational semantics and the operational
semantics defined above are ``in sync'', as expressed by soundness and
computational adequacy (cf.\ Section~\ref{background:models}).

\begin{therm}[Soundness]
  Let $s$ and $t$ be any PCF terms of PCF type $\sigma$. If
  $s \smallstep^\ast t$, then $\densem{s} = \densem{t}$.
\end{therm}
\begin{proof}[\coqidentproof{Partiality}{PCF}{soundness}]
  Since the carriers of dcpos are defined to be sets, the type
  \(\densem{s} = \densem{t}\) is a proposition. Therefore, we can use induction
  on the derivation of \(s \smallstep^\ast t\). We use the Kleisli monad laws in
  proving some of the cases. For example, one step is to prove that
  \[
    \densem{\ifz s\,t\,\underline{n+1}} = \densem{t}.
  \]
  This may be proved by the following chain of equalities:
  \begin{alignat*}{2}
    \densem{\ifz s\,t\,\underline{n+1}} &=
    \densem{\ifz s\,t}(\densem{\underline{n+1}}) \\
    &= \densem{\ifz s\,t}(\eta(n+1))\quad&&\text{(by
      Proposition~\ref{densemnumeral})} \\
    &= (\chi_{\densem{s},\densem{t}})^\#(\eta(n+1)) \quad&&\text{(by definition
      of $\densem{\ifz}$)} \\
    &= \chi_{\densem{s},\densem{t}}(n+1) \quad&&\text{(by
      Theorem~\ref{lifting-is-monad})} \\ &= \densem{t}.\tag*{\qedhere}
  \end{alignat*}
\end{proof}

Ideally, we would like a converse to soundness. However, this is not possible,
as for example,
${\densem{\PCFk{\zeroo}}} = {\densem{\PCFk (\succc (\pred \zeroo))}}$, but
neither ${\PCFk \zeroo} \smallstep^\ast {\PCFk (\succc (\pred \zeroo))}$ nor
${\PCFk (\succc (\pred \zeroo))} \smallstep^\ast {\PCFk \zeroo}$ holds. We do,
however, have the following.

\begin{therm}[Computational adequacy]
  \label{Adequacy}
  Let $t$ be a PCF term of PCF type $\iota$. Then,
  \[
    \prod_{p : \isdefined (\densem{t})} t \smallstep^\ast \underline{\liftvalue
      (\densem{t}) (p)}.
  \]
  Equivalently, for every $n : \natt$, it holds that
  $ \densem{t} = \densem{\underline{n}} $ implies
  $t \smallstepclos {\underline n}$.
\end{therm}

We do not prove computational adequacy directly, as, unlike soundness, it does
not allow for a straightforward proof by induction. Instead, we use the standard
technique of logical relations \cite[Chapter 7]{Streicher2006} and obtain the
result as a direct corollary of Lemma~\ref{mainlemma}.

\begin{definition}[\coqident{Partiality}{PCF}{adequacy_relation}]
  For every PCF type $\sigma$, define a relation
  \[
    R_\sigma : \textup{PCF terms of type $\sigma$} \to \densem{\sigma} \to \Prop
  \]
  by induction on $\sigma$:
  \begin{enumerate}
  \item
    $t R_\iota d \colonequiv \prod_{p : \isdefined (d)} t \smallstep^\ast
    \underline {\liftvalue(d)(p)}$;
  \item
    $s R_{\tau \Rightarrow \rho} f \colonequiv \prod_{t : \text{PCF terms of
        type $\tau$}} \prod_{d : \densem{\tau}} \pa*{t R_\tau d \to st R_\rho
      f(d)}$.
  \end{enumerate}
  We sometimes omit the type subscript $\sigma$ in $R_\sigma$.
\end{definition}

\begin{lemma}
  \label{logicalrelationsmall-step}
  Let $s$ and $t$ be PCF terms of type $\sigma$ and let $d$ be an element
  of~$\densem{\sigma}$. If $s \smallstep^\ast t$ and $t R_\sigma d$, then
  $s R_\sigma d$.
\end{lemma}
\begin{proof}[\coqidentproof{Partiality}{PCF}{adequacy_step}]
  By induction on $\sigma$, making use of the last part of
  Lemma~\ref{refltransreduce}.
\end{proof}

\begin{lemma}
  \label{logicalrelationconstants}
  For $t$ equal to
  $\zeroo, {\succc}, {\pred}, {\ifz}, {\PCFk}$ or
  ${\PCFs}$, we have: $t R \densem{t}$.
\end{lemma}
\begin{proof}
  [Proof \textnormal{(\coqident{Partiality}{PCF}{adequacy_zero},
  \coqident{Partiality}{PCF}{adequacy_succ},
  \coqident{Partiality}{PCF}{adequacy_pred},
  \coqident{Partiality}{PCF}{adequacy_ifz},
  \coqident{Partiality}{PCF}{adequacy_k},\\
  \coqident{Partiality}{PCF}{adequacy_s})}.]
  By the previous lemma and Lemma~\ref{refltransreduce}.
\end{proof}

Next, we wish to extend the previous lemma the case where
$t \equiv {\fix_\sigma}$ for any PCF type $\sigma$. This is slightly more
complicated and we need two intermediate lemmas. Only the second requires a
non-trivial proof.

\begin{lemma}
  \label{logicalrelationbottom}
  Let $\sigma$ be a PCF type and let $\bot$ be the least element of
  $\densem{\sigma}$. Then, $t R_\sigma \bot$ for any PCF term $t$ of type
  $\sigma$.
\end{lemma}
\begin{proof}[\coqidentproof{Partiality}{PCF}{adequacy_bottom}]
  By induction on $\sigma$. For the base type, this holds vacuously. For
  function types, it follows by induction hypothesis and the pointwise ordering.
\end{proof}

\begin{lemma}
  The logical relation is closed under directed suprema. That is, for every PCF
  term $t$ of type $\sigma$ and every directed family
  $d : I \to \densem{\sigma}$, if $t R_\sigma d_i$ for every $i : I$, then
  $t R_\sigma \bigsqcup_{i : I}d_i$.
\end{lemma}
\begin{proof}[\coqidentproof{Partiality}{PCF}{adequacy_lubs}]
  This proof is somewhat different from the classical proof, so we spell out the
  details. We prove the lemma by induction on $\sigma$.

  The case when $\sigma$ is a function type is easy, because least upper bounds
  are calculated pointwise and so it reduces to an application of the induction
  hypothesis. We concentrate on the case when $\sigma \equiv \iota$ instead.

  Recall that $\bigsqcup_{i : I} d_i$ is given by
  $\pa*{\squash*{\sum_{i : I} \isdefined(d_i)}, \phi}$, where $\phi$ is the
  factorisation of
  \[
    \sum_{i : I} \isdefined(d_i) \to \lift(\natt), \quad (i, p_i) \mapsto
    \liftvalue(d_i)(p_i)
  \]
  through $\squash*{\sum_{i : I} \isdefined(d_i)}$.

  We are tasked with proving that $t \smallstep^\ast \underline{\phi (p)}$ for
  every $p : \isdefined\pa*{\bigsqcup_{i : I}d_i}$. So assume that
  $p : \squash*{\sum_{i : I}\isdefined(d_i)}$. Since we are trying to prove a
  proposition (as $\smallstep^\ast$ is proposition-valued), we may actually
  assume that we have $(j,p_{j}) : \sum_{i : I}\isdefined(d_i)$.  By definition
  of $\phi$ we have: $\phi (p) = \liftvalue(d_{j})(p_{j})$ and by assumption we
  know that $t \smallstep^\ast \underline{\liftvalue (d_{j}) (p_{j})}$, so we
  are done.
\end{proof}

\begin{lemma}
  For every PCF type $\sigma$, we have
  $\fix_\sigma R_{(\sigma \Rightarrow \sigma)\Rightarrow \sigma}
  \densem{\fix_\sigma}$.
\end{lemma}
\begin{proof}[\coqidentproof{Partiality}{PCF}{adequacy_fixp}]
  Let $t$ be a PCF term of type $\sigma\Rightarrow\sigma$ and let
  $f : \densem{\sigma\Rightarrow\sigma}$ such that
  $t R_{\sigma\Rightarrow\sigma} f$. We are to prove that
  $\fix t R_\sigma \mu(f)$.

  By definition of $\mu$ and the previous lemma, it suffices to prove that
  $\fix t R_\sigma f^n(\bot)$ where $\bot$ is the least element of
  $\densem{\sigma}$ for every natural number $n$. We do so by induction on $n$.

  The base case is an application of Lemma~\ref{logicalrelationbottom}.

  Now suppose that $\fix t R_\sigma f^m(\bot)$. Then, using
  $t R_{\sigma\Rightarrow\sigma} f$, we find:
  $t(\fix t) R_\sigma f(f^m(\bot))$. Hence, by
  Lemma~\ref{logicalrelationsmall-step}, we obtain the desired
  $\fix t R_\sigma f^{m+1}(\bot)$, completing our proof by induction.
\end{proof}

\begin{lemma}[Fundamental Theorem]
  \label{mainlemma}
  For every PCF term $t$ of type $\sigma$, we have $t R_\sigma \densem{t}$.
\end{lemma}
\begin{proof}[\coqidentproof{Partiality}{PCF}{adequacy_allterms}]
  The proof is by induction on $t$. The base cases are taken care of by
  Lemma~\ref{logicalrelationconstants} and the previous lemma. For the inductive
  step, suppose $t$ is a PCF term of type $\sigma\Rightarrow\tau$. By induction
  hypothesis, $ts R_\tau \densem{ts}$ for every PCF term $s$ of type $\sigma$,
  but $\densem{ts} \equiv \densem{t}\densem{s}$, so we are done.
\end{proof}

Computational adequacy is now a direct corollary of Lemma~\ref{mainlemma}.
\begin{proof}[Proof of computational adequacy
  \textnormal{(\coqident{Partiality}{PCF}{adequacy},
    \coqident{Partiality}{PCF}{adequacy_alt},
    \coqident{Partiality}{PCF}{alt_adequacy})}]
  Take $\sigma$ to be the base type $\iota$ in Lemma~\ref{mainlemma}.
\end{proof}

\paragraph*{Using computational adequacy to compute.} An interesting use of
computational adequacy is that it allows one to argue semantically to obtain
results about termination (i.e.\ reduction to a numeral) in PCF. Classically,
every PCF program of type $\iota$ either terminates or it does not.  From a
constructive point of view, we wait for a program to terminate, with no a priori
knowledge of termination. The waiting could be indefinite. Less naively, we
could limit the number of computation steps to avoid indefinite waiting, with an
obvious shortcoming: how many steps are enough? Instead, one could use
computational adequacy to compute as follows.

Let $\sigma$ be a PCF type. A \emph{functional of type $\sigma$} is an element
of $\densem{\sigma}$. By induction on PCF types, we define when a functional is
said to be \emph{total}:
\begin{enumerate}
\item a functional $i$ of type $\iota$ is total if $i = \densem{\underline n}$
  for some natural number $n$;
\item a functional $f$ of type $\sigma\Rightarrow\tau$ is total if it maps total
  functionals to total functionals, viz.\ $f(d)$ is a total functional of type
  $\tau$ for every total functional $d$ of type $\sigma$.
\end{enumerate}
Now, let $s$ be a PCF term of type
$\sigma_1 \Rightarrow \sigma_2 \Rightarrow \dots \Rightarrow \sigma_n
\Rightarrow \iota$. If we can prove that $\densem{s}$ is total, then
computational adequacy lets us conclude that for all total inputs
$\densem{t_1} : \densem{\sigma_1},\dots,\densem{t_n} : \densem{\sigma_n}$, the
term $s(t_1,\dots,t_n)$ reduces to the numeral representing
$\densem{s}(\densem{t_1},\dots,\densem{t_n})$. Thus, the semantic proof of
totality plays the role of ``enough steps''. Of course, this still requires us
to prove that \(\densem{s}\) is total, which may be challenging. But the point
is that we can use domain-theoretic arguments to prove this about the denotation
\(\densem{s}\), whereas in a direct proof of termination we would only have the
operational semantics available for our argument.

\section{Semidecidable propositions and PCF terms of base type}
\label{sec:charprop}
In this section we characterise those propositions that arise from the PCF
interpretation, in the following sense. Every PCF term $t$ of base type $\iota$
gives rise to a proposition via the Scott model, namely
$\isdefined(\densem{t})$. We wish to show that such propositions are
semidecidable, which we define now. For ease of notation, we write $\exists$ for
the propositional truncation of $\Sigma$.

\begin{definition}\label{def:semidecidable}
  A proposition \(Q\) is \emph{semidecidable} if it is equivalent to
  \(\exists_{n_1 : \natt}\cdots\exists_{n_k : \natt}P(n_1,\dots,n_k)\) where
  \(k\) is some natural number and \(P : \natt^k \to \Omega\) is a
  proposition-valued family such that \(P(m_1,\dots,m_k)\) is decidable for every
  \((m_1,\dots,m_k) : \natt^k\).
\end{definition}

We will prove our goal that \(\isdefined(\densem{t})\) is semidecidable by
showing that it is logically equivalent to
$\exists_{n : \natt} \exists_{k : \natt} \,t \smallstep^k {\underline n}$ and by
proving that $t \smallstep^k {\underline n}$ is decidable. Here
$t\smallstep^k{\underline n}$ says that $t$ reduces to $\underline n$ in at most
$k$ steps. A first step towards this is the following, which is a consequence of
soundness and computational adequacy.

\begin{lemma}
  \label{soundnesscompadequacyprop}
  Let $t$ be a PCF term of type $\iota$. We have the following logical
  equivalences
  \[
    \isdefined(\densem{t}) \enspace\longleftrightarrow\enspace \sum_{n : \natt}
    t \smallstep^\ast {\underline{n}} \enspace\longleftrightarrow\enspace
    \squash*{\sum_{n : \natt} t \smallstep^\ast \underline {n}}.
  \]
\end{lemma}
\begin{proof}[\coqidentproof{Partiality}{PCF}{char_pcf_propositions}]
  We start by proving the first logical equivalence. The second then follows
  from the fact that $\isdefined(\densem{t})$ is a proposition. Suppose $p$ is
  of type $\isdefined(\densem{t})$. By computational adequacy, we find that
  $t \smallstep^\ast {\underline {\liftvalue (\densem{t}) (p)}}$, so we are
  done.

  Conversely, suppose that we are given a natural number $n$ such that
  $t \smallstep^\ast \underline n$. Soundness and
  Proposition~\ref{densemnumeral} then yield $\densem{t} = \eta(n)$. Now
  $\star : \isdefined(\eta(n))$, so we may transport along the equality to get
  an element of $\isdefined(\densem{t})$.
\end{proof}

In order to characterise the propositions arising from PCF terms of base type as
semidecidable, we wish to prove that $t \smallstep^\ast \underline{n}$ is
semidecidable for every PCF term $t$ of type $\iota$ and natural number~$n$. We
do so by proving some more general results, which we present in
Section~\ref{semidec-rel} and Section~\ref{indexedWtypes}. Here, we outline our
general strategy and highlight the main theorems and their applications to the
problem at hand.

Given any (proposition-valued) relation $R$ on a type $X$, we can define the
$k$-step reflexive transitive closure $R^k$ of $R$ and prove that $x R^* y$ if
and only if $\exists_{k : \natt} x R^k y$. Thus we obtain the following
(intermediate) result.
\begin{lemma}
  For every PCF term $t$ of type $\iota$, we have:
  \[
    \isdefined\pa*{\densem{t}} \longleftrightarrow \exists_{n : \natt}\exists_{k
      : \natt} t \smallstep^k {\underline n}.
  \]
\end{lemma}
\begin{proof}[\coqidentproof{Partiality}{PCF}{char_pcf_propositions'}]
  This follows from Lemma~\ref{soundnesscompadequacyprop} and
  Lemma~\ref{refltransassumstep}.
\end{proof}

Thus, to prove that $s \smallstepclos t$ is semidecidable, it suffices to show
that $s \smallstep^k t$ is decidable for every natural number $k$. To this end,
we prove the following in Section~\ref{semidec-rel}.
\begin{theorem}[Theorem~\ref{decidablekstep}]
  Let $R$ be relation on a type $X$. If
  \begin{enumerate}
  \item\label{assump-1} $X$ has decidable equality;
  \item\label{assump-2} $R$ is single-valued;
  \item\label{assump-3} $\sum_{y : X} x R y$ is decidable for every $x : X$;
  \end{enumerate}
  then, the $k$-step reflexive transitive closure $R^k$ of $R$ is decidable for
  every natural number $k$.
\end{theorem}

Thus, $s \smallstep^k t$ is decidable if it satisfies the assumptions
\eqref{assump-1}--\eqref{assump-3}. Assumptions~\eqref{assump-2}
and~\eqref{assump-3} can be verified by inspection of the small-step operational
semantics once~\eqref{assump-1} has been proved.

Hence, we are to prove that the type of PCF terms has decidable equality. This
can be done fairly directly by induction (as pointed out by one of the anonymous
referees). However, we take it as an opportunity to study (in
Section~\ref{indexedWtypes}) a more general and powerful result on indexed
\(\mathsf{W}\)-types (see Theorem~\ref{indexedWtypedeceq}), which is interesting
in its own right. For now, we take it as proved that the PCF terms have
decidable equality and continue our study of propositions coming from PCF terms
at the base type.

\begin{therm}\label{semidecidablepcfprops}
  The propositions that arise from PCF terms $t$ of type $\iota$ are all
  semidecidable, as witnessed by the following logical equivalence:
  \[
    \isdefined(\densem{t}) \longleftrightarrow \exists_{n : \natt} \exists_{k :
      \natt}\, t \smallstep^k \underline n
  \]
  and the decidability of $t \smallstep^k \underline n$.
\end{therm}

Given this theorem, it is natural to ask whether we can
construct the Scott model of PCF using a restricted version of the lifting
monad. Write $\Omega_{\textup{sd}}$ for the type of propositions that are
semidecidable. Theorem~\ref{semidecidablepcfprops} says that the map
\begin{align*}
  \text{PCF terms of type $\iota$} &\to \Omega \\
  t &\mapsto \isdefined\pa*{\densem{t}}
\end{align*}
factors through $\Omegasd$. Thus, could we also have constructed the Scott model
of PCF using the restricted lifting
$\liftsd(X) \colonequiv \sum_{P : \Omegasd} (P \to X)$?

Of course, $\liftsd(X)$ is not a dcpo, because, recalling our construction of
suprema in $\lift(X)$, given a directed family $u : I \to \liftsd(X)$, the
proposition $\squash*{\sum_{i : I} \isdefined\pa*{u_i}}$ need not be
semidecidable. However, one might think that $\liftsd(X)$ still has suprema of
$\natt$-indexed directed families (which would suffice for the Scott model), but
proving this requires an instance of the axiom of countable choice, cf.\
\cite[Theorem 5.34]{Knapp2018} and \cite[Theorem 5]{EscardoKnapp2017}. Moreover,
$\liftsd$ is a monad if and if only a particular choice principle (which is
implied by countable choice) holds, see \cite[Theorem 3]{EscardoKnapp2017} and
\cite[Section 5.8]{Knapp2018}. In fact, this choice principle is the one
discussed in Section~\ref{sec:relatedwork}; \cite[Theorem 5.28]{Knapp2018}
proves that if $X$ is a set then $\liftsd(X)$ is equivalent to the quotiented
delay monad.

Again, as pointed out in Section~\ref{sec:relatedwork}, the problem is that this
choice principle cannot be proved in constructive univalent type theory.

\subsection{Decidability of the \texorpdfstring{$k$}{k}-step reflexive
  transitive closure of a relation}
\label{semidec-rel}
In this section we provide sufficient conditions on a relation for its $k$-step
reflexive transitive closure to be decidable. The purpose of this section is to
prove Theorem~\ref{decidablekstep}, whose use we have explained above.

\begin{definition}[\coqident{Foundations}{Sets}{hrel}]
  A \emph{relation on $X$} is a term of type $X \to X \to \Prop$.
\end{definition}
\begin{definition}
  [\coqident{MoreFoundations}{ClosureOfHrel}{refltransclos_step},
  \coqident{MoreFoundations}{ClosureOfHrel}{refltransclos_step_hrel}] Let $R$ be
  a relation on a type $X$. We wish to define the $k$-step reflexive transitive
  closure of $R$. As in Definition~\ref{def:refltransclos}, we want this to be
  proposition-valued again. Therefore, we proceed as follows. For any natural
  number $k$, define $x R_{k} y$ by induction on $k$:
  \begin{enumerate}
  \item $x R_{0}y \colonequiv x = y$;
  \item $x R_{k+1}z \colonequiv \sum_{y : X} x R y \times y R_{k} z$.
  \end{enumerate}
  The \emph{$k$-step reflexive transitive closure $R^k$ of $R$} is now defined
  as the relation on $X$ given by ${x R^k y \colonequiv \squash{ x R_k y}}$.
\end{definition}

We wish to prove that $x R^\ast y$ if and only if
$\squash*{\sum_{k : \natt} x R^k y}$. The following lemma is the first step
towards that.

\begin{lemma}
  Let $R$ be a relation on $X$. Recall the untruncated reflexive transitive
  closure $R_\ast$ from Definition~\ref{def:refltransclos}. We have a logical
  equivalence for every $x,y$ in $X$\textup{:}
  \[
    x R_\ast y \longleftrightarrow \sum_{k : \natt} x R_{k} y.
  \]
\end{lemma}
\begin{proof}
  [Proof \textnormal{(\coqident{MoreFoundations}{ClosureOfHrel}{stepleftequiv},
    \coqident{MoreFoundations}{ClosureOfHrel}{left_regular_equiv})}]
  Define $x R' y$ inductively by:
  \begin{alignat*}{2}
    &\mathsf{refl'} &&: \prod_{x : X} x R' x; \\
    &\mathsf{left} &&: \prod_{x y z : X} x R y \to y R' z \to x R' z.
  \end{alignat*}
  It is not hard to verify that $R'$ is reflexive, transitive and that it
  extends $R$. Using this, one shows that $x R' y$ and $x R_\ast y$ are
  logically equivalent for every $x,y : X$. Now one easily proves
  $\prod_{k : \natt} (x R_k y \to x R' y)$ by induction on $k$. This yields
  $\pa*{\sum_{k : \natt} x R_k y} \to x R' y$. The converse is also easily
  established. Thus, $x R' y$ and $\sum_{k : \natt} x R_k y$ are logically
  equivalent, finishing the proof.
\end{proof}

The next lemma extends the previous one to the propositional truncations.
\begin{lemma}
  \label{refltransassumstep}
  Let $R$ be a relation on $X$. For every $x,y : X$, we have a logical
  equivalence:
  \[
    x R^\ast y \longleftrightarrow \squash*{\sum_{k : \natt} x R^k y}.
  \]
\end{lemma}
\begin{proof}
  [Proof
  \textnormal{(\coqident{MoreFoundations}{ClosureOfHrel}{stepleftequiv_hrel},
    \coqident{MoreFoundations}{ClosureOfHrel}{left_regular_equiv})}.]
  Let $x$ and $y$ be in $X$. By the previous lemma and functoriality of
  propositional truncation, we have
  \[
    x R^\ast y \equiv \squash{x R_\ast y} \longleftrightarrow \squash*{\sum_{k :
        \natt} x R_k y}.
  \]
  But the latter is equivalent to
  $\squash*{\sum_{k : \natt} \squash*{x R_k y}} \equiv \squash*{\sum_{k : \natt}
    x R^k y}$ by \cite[Theorem 7.3.9]{HoTTbook}. This may also be proved
  directly, as done in the formalisation.
\end{proof}

\begin{definition}[\coqident{MoreFoundations}{ClosureOfHrel}{is_singlevalued}]
  A relation $R$ on $X$ is said to be \emph{single-valued} if for every
  $x,y,z : X$ with $x R y$ and $x R z$ we have $y = z$.
\end{definition}

\begin{definition}[\coqident{MoreFoundations}{ClosureOfHrel}{isdecidable_hrel}]
  A relation $R$ on $X$ is said to be \emph{decidable} if the type $x R y$ is
  decidable for every $x$ and $y$ in $X$.
\end{definition}

\begin{lemma}
  \label{decidablesquash}
  Let $X$ be a type. If $X$ is decidable, then so is $\squash{X}$.
\end{lemma}
\begin{proof}[\coqidentproof{Foundations}{Propositions}{decidable_ishinh}]
  Suppose that $X$ is decidable. Then there are two cases to consider. Either we
  have $x : X$ or $\lnot X$. If we have $x : X$, then obviously we have
  $|x| : \squash{X}$.

  So suppose that $\lnot X$. We claim that $\lnot\squash{X}$. Assuming
  $\squash{X}$, we must find a term of type~$\emptyt$. But $\emptyt$ is a
  proposition, so we may actually assume that we have $x : X$. Using $\lnot X$,
  we then obtain~$\emptyt$, as desired.
\end{proof}

\begin{therm}
  \label{decidablekstep}
  Let $R$ be relation on a type $X$. If
  \begin{enumerate}
  \item\label{cond-1} $X$ has decidable equality;
  \item\label{cond-2} $R$ is single-valued;
  \item\label{cond-3} $\sum_{y : X} x R y$ is decidable for every $x : X$;
  \end{enumerate}
  then, the $k$-step reflexive transitive closure $R^k$ of $R$ is decidable for
  every natural number $k$.
\end{therm}
\begin{proof}[\coqidentproof{MoreFoundations}{ClosureOfHrel}{decidable_step}]
  Suppose $X$ and $R$ satisfy conditions \eqref{cond-1}--\eqref{cond-3}. By
  Lemma~\ref{decidablesquash}, it suffices to prove that the untruncated version
  of $R^k$, that is $R_k$, is decidable by induction on $k$.

  For the base case, let $x$ and $y$ be elements of $X$. We need to decide
  $x R_0 y$. By definition this means deciding $x = y$, which we can, since $X$
  is assumed to have decidable equality.

  Now suppose $x$ and $z$ are elements of $X$ and that $a R_k b$ is decidable
  for every $a,b : X$. We need to show that $xR_{k+1}z$ is decidable. By
  definition this means that we must prove
  \begin{equation*}\label{toprovedec}
    \sum_{y : X} x R y \times y R_k z \tag{$\ast$}
  \end{equation*}
  to be decidable. By~\eqref{cond-3}, we can decide $\sum_{y : X}x R y$. Obviously, if we
  have $\lnot\sum_{y : X}x R y$, then $\lnot\eqref{toprovedec}$. So assume that
  we have $y : X$ such that $x R y$. By induction hypothesis, $y R_k z$ is
  decidable. If we have $y R_k z$, then we get $\eqref{toprovedec}$. So suppose
  that $\lnot y R_k z$. We claim that $\lnot\eqref{toprovedec}$. For suppose
  $\eqref{toprovedec}$, then we obtain $y' : X$ with $x R y'$ and $y' R_k
  z$. But $R$ is single-valued, so $y = y'$ and hence, $y R_k z$, contradicting
  our assumption.
\end{proof}

\subsection{Decidable equality and indexed
  \texorpdfstring{$\mathsf{W}$}{W}-types}
\label{indexedWtypes}

We wish to prove that a certain class of indexed \(\mathsf{W}\)-types has
decidable equality. Indexed \(\mathsf{W}\)-types are a generalisation of
$\mathsf{W}$-types that allows for many-sorted terms. One may consult
\cite[Section 5.3]{HoTTbook} for an explanation of regular
$\mathsf{W}$-types. The PCF terms form a natural example of an indexed
$\mathsf{W}$-type, where the sorts will be the formal types of PCF terms. We
apply the general result for indexed \(\mathsf{W}\)-types to see that the PCF
terms have decidable equality.

\subsubsection{PCF terms as an indexed \texorpdfstring{$\mathsf{W}$}{W}-type}
In this section we explain what indexed $\mathsf{W}$-types are and how PCF terms
can encoded as such an indexed $\mathsf{W}$-type.
\begin{definition}[\coqident{MoreFoundations}{Wtypes}{indexedWtype}]
  Let $A$ and $I$ be types and let $B$ be a type family over $A$. Suppose we
  have $t : A \to I$ and $s : \pa*{\sum_{a : A} B(a)} \to I$. The \emph{indexed
    $\mathsf W$-type $\mathsf{W}_{s,t}$ specified by $s$~and~$t$} is the
  inductive type family over $I$ generated by the following constructor:
  \[
    \mathsf{indexedsup} : \prod_{a : A}\big(B(a) \to
    \mathsf{W}_{s,t}(s(a,b))\big) \to \mathsf{W}_{s,t}(t(a)).
  \]
  We have the following induction principle for indexed $\mathsf{W}$-types. If
  $E : \prod_{i : I}\pa*{\mathsf{W}_{s,t}(i) \to \mathcal U}$, then to prove
  $\prod_{i : I}\prod_{w : \mathsf{W}_{s,t}(i)} E(i,w)$, it suffices to show
  that for any $a : A$ and $f : \prod_{b:B(a)}\mathsf{W}_{s,t}(s(a,b))$
  satisfying $E(s(a,b),f(b))$ for every $b : B(a)$ (the \emph{induction
    hypothesis}), we have a term of type $E(t(a),\mathsf{indexedsup}(a,f))$.
\end{definition}

Just as with regular $\mathsf{W}$-types, we can think of indexed
$\mathsf{W}$-types as encoding a particular class of inductive types. In this
interpretation, $A$ encodes the constructors of the inductive type, whereas $B$
encodes the arity of each constructor. However, each constructor has a ``sort''
given by $t(a) : I$. Given a constructor $a : A$ and a label of an argument
$b : B(a)$, the sort of this argument is given by $s(a,b)$.

\begin{example}\label{termsasindexedWtype}
  In this example, we show that a fragment of the PCF terms can be encoded as an
  indexed $\mathsf{W}$-type. One could extend the encoding to capture all PCF
  terms, but we do not spell out the tedious details here, as a fragment
  suffices to get the idea across.

  The type family $\mathsf{T}$ is inductively defined as:
  \begin{enumerate}
  \item $\mathsf{zero}$ is a term of type $\iota$;
  \item $\mathsf{succ}$ is a term of type $\iota \Rightarrow \iota$;
  \item for every PCF type $\sigma$ and $\tau$, we have a term
    $\mathsf{app}_{\sigma,\tau}$ of type
    $(\sigma \Rightarrow \tau) \Rightarrow \sigma \Rightarrow \tau$.
  \end{enumerate}
  We can encode $\mathsf{T}$ as an indexed $\mathsf{W}$-type. Let us write
  $\twot$ for $\unitt + \unitt$ and $0_{\twot}$ and $1_{\twot}$
  for its elements. Take $I$ to be the type of PCF types and put
  $A \colonequiv \twot + (I \times I)$. Define $B : A \to \mathcal U$ by
  \[
    B(\inl(0_{\twot})) \colonequiv B(\inl(1_{\twot}))
    \colonequiv \emptyt\quad\text{and}\quad B(\inr(\sigma,\tau))
    \colonequiv \twot.
  \]
  Finally, define $t$ by
  \[
    t(\inl(0_{\twot})) \colonequiv \iota;\quad
    t(\inl(1_{\twot})) \colonequiv \iota \Rightarrow \iota;\quad
    \text{and}\quad t(\inr(\sigma,\tau)) \colonequiv \tau;
  \]
  and $s$ by
  \[
    s(\inr(\sigma,\tau),0_{\twot}) \colonequiv \sigma \Rightarrow
    \tau;\quad\text{and}\quad s(\inr(\sigma,\tau),1_{\twot})
    \colonequiv \sigma;
  \]
  on the other elements $s$ is defined as the unique function from $\emptyt$.

  One can check that given a PCF type $\sigma : I$, there is a type equivalence
  $T(\sigma) \simeq W_{s,t}(\sigma)$.
\end{example}

\subsubsection{Indexed \texorpdfstring{$\mathsf{W}$}{W}-types with decidable equality}
We wish to isolate some conditions on the parameters of an indexed
$\mathsf{W}$-type that are sufficient to conclude that an indexed
$\mathsf{W}$-type has decidable equality. We first need a few definitions before
we can state the theorem.

\begin{definition}[\texttt{WeaklyCompactTypes} in \cite{TypeTopology},
  \coqident{MoreFoundations}{Wtypes}{picompact}]
  \label{picompact}
  A type $X$ is called \emph{$\Pi$-compact} when every type family $Y$ over $X$
  satisfies: if $Y(x)$ is decidable for every $x : X$, then so is the dependent
  product $\prod_{x : X} Y(x)$.
\end{definition}

\begin{example}
  [\coqident{MoreFoundations}{Wtypes}{picompact_empty},
  \coqident{MoreFoundations}{Wtypes}{picompact_unit}]
  \label{emptyunitpicompact}
  The empty type $\emptyt$ is vacuously $\Pi$-compact. The unit type $\unitt$ is
  also easily seen to be $\Pi$-compact. There are also interesting examples of
  infinite types that are $\Pi$-compact, such as $\natt_{\infty}$, the one-point
  compactification of the natural numbers
  \cite[\texttt{WeaklyCompactTypes}]{TypeTopology}.
\end{example}

We are now in position to state the general theorem about decidable equality on
indexed $\mathsf{W}$-types.

\begin{therm}
  \label{indexedWtypedeceq'}
  Let $A$ and $I$ be types and $B$ a type family over $A$. Suppose $t : A \to I$
  and $s : \pa*{\sum_{a : A}B(a)} \to I$. If $A$ has decidable equality, $B(a)$
  is $\Pi$-compact for every $a : A$ and $I$ is a set, then
  $\mathsf{W}_{s,t}(i)$ has decidable equality for every $i : I$.
\end{therm}

The proof of Theorem~\ref{indexedWtypedeceq'} is quite technical, so we postpone
it until Section~\ref{proofofindexedWtypedeceq}. Instead, we next describe how
to apply the theorem to prove that the PCF terms have decidable equality.

\subsubsection{PCF terms have decidable equality}
In this section we show that the PCF terms have decidable equality by applying
Theorem~\ref{indexedWtypedeceq'}.  Before we proceed, we record some useful
lemmas.
\begin{lemma}
  \label{decidableiff}
  Let $X$ and $Y$ be logically equivalent types. The type $X$ is decidable if
  and only if $Y$ is decidable.
\end{lemma}
\begin{proof}[\coqidentproof{MoreFoundations}{Wtypes}{decidable_iff}]
  Straightforward.
\end{proof}

\begin{definition}
  A type $X$ is called a \emph{retract} of a type $Y$ if there are maps
  $s : X \to Y$ (the~\emph{section}) and $r : Y \to X$ (the \emph{retraction})
  such that $\prod_{x : X} r(s(x)) = x$.
\end{definition}
\begin{lemma}
  \label{retractdeceq}
  Let $X$ be a retract of $Y$. If $Y$ has decidable equality, then so does $X$.
\end{lemma}
\begin{proof}[\coqidentproof{MoreFoundations}{Wtypes}{isdeceq_retract}]
  Let $r : Y \to X$ and $s : X \to Y$ be respectively the retraction and section
  establishing $X$ as a retract of $Y$. Let $a,b : X$. Since $Y$ has decidable
  equality, we can consider two cases: $r(a) = r(b)$ and $r(a) \neq r(b)$.
  In the first case, we find $a = s(r(a)) = s(r(b)) = b$. In the second case, we
  immediately see that $a \neq b$. This finishes the proof.
\end{proof}

\begin{lemma}
  \label{coproductpicompact}
  The $\Pi$-compact types are closed under binary coproducts.
\end{lemma}
\begin{proof}[\coqidentproof{MoreFoundations}{Wtypes}{picompact_coprod}]
  Let $X$ and $Y$ be $\Pi$-compact types. Suppose $F$ is a type family over
  $X + Y$ such that $F(z)$ is decidable for every $z : X + Y$. We must show that
  $\prod_{z : X + Y}F(z)$ is decidable.

  Define $F_X : X \to \mathcal U$ by $F_X(x) \colonequiv F(\inl(x))$ and
  $F_Y : Y \to\mathcal U$ as $F_Y(y) \colonequiv F(\inr(y))$. By our
  assumption on $F$, the types $F_X(x)$ and $F_Y(y)$ are decidable for every
  $x : X$ and $y : Y$. Hence, since $X$ and $Y$ are assumed to be $\Pi$-compact,
  the dependent products $\prod_{x : X}F_X(x)$ and $\prod_{y : Y}F_Y(y)$ are
  decidable.

  Finally, $\prod_{z : X + Y}F(z)$ is logically equivalent to
  $\prod_{x : X}F_X(x) \times \prod_{y : Y}F_Y(y)$. Since the product of two
  decidable types is again decidable, an application of Lemma~\ref{decidableiff}
  now finishes the proof.
\end{proof}

Finally, let us see how to apply Theorem~\ref{indexedWtypedeceq'} to see that
the PCF terms have decidable equality.
\begin{therm}
  The PCF terms have decidable equality.
\end{therm}
\begin{proof}
  As with Example~\ref{termsasindexedWtype}, we only spell out the details for
  the fragment $\mathsf{T}$. Recall that $\mathsf{T}$ may be encoded as a
  $\mathsf{W}$-type, indexed by the PCF types. Using
  Example~\ref{emptyunitpicompact} and Lemma~\ref{coproductpicompact}, we see
  that $B(a)$ is $\Pi$-compact for every $a : A$. Note that $A$ has decidable
  equality if $I$ does. So it remains to prove that $I$, the type of PCF types,
  has decidable equality.

  This will be another application of Theorem~\ref{indexedWtypedeceq'}. Define
  $A' \colonequiv \twot$ and define $B' : A' \to \mathcal U$ by
  $B'(\inl(\star)) \colonequiv \emptyt$ and $B'(\inr(\star)) \colonequiv
  \twot$. Let $t'$ and $s'$ be the unique functions to $\unitt$ from $A'$ and
  $\sum_{x : A'}B'(x)$, respectively. One quickly verifies that the type of PCF
  types is a retract of $\mathsf{W}_{s',t'}(\star)$. Observe that $B'(x)$ is
  $\Pi$-compact for every $x : A'$ because of Example~\ref{emptyunitpicompact}
  and Lemma~\ref{coproductpicompact}.  Finally, $\unitt$ and $A' \equiv 2$
  clearly have decidable equality, so by Theorem~\ref{indexedWtypedeceq'} the
  type $\mathsf{W}_{s',t'}(\star)$ has decidable equality. Thus, by
  Lemma~\ref{retractdeceq}, so do the PCF types.
\end{proof}

\subsubsection{Proof of Theorem~\texorpdfstring{\ref{indexedWtypedeceq'}}{Theorem 66}}
\label{proofofindexedWtypedeceq}

In this section we prove Theorem~\ref{indexedWtypedeceq'} by deriving it as a
corollary of another result, namely Theorem~\ref{indexedWtypedeceq} below. This
result seems to have been first established by Jasper Hugunin, who reported on
it in a post on the Homotopy Type Theory mailing
list~\cite{HuguninMail2017}. Our proof of Theorem~\ref{indexedWtypedeceq} is a
simplified written-up account of Hugunin's Coq code
\cite[\texttt{FiberProperties.v}]{Hugunin2017}.

\begin{definition}
  [Definition 2.4.2 in \cite{HoTTbook}, \coqident{Foundations}{PartA}{hfiber}]
  Let $f : X \to Y$ be a map. The \emph{fiber of $f$ over a point $y : Y$} is
  \[
    \mathsf{fib}_f(y) \colonequiv \sum_{x : X}(f(x) = y).
  \]
\end{definition}

\begin{therm}
  [Jasper Hugunin]
  \label{indexedWtypedeceq}
  Let $A$ and $I$ be types and $B$ a type family over $A$. Suppose $t : A \to I$
  and $s : \pa*{\sum_{a : A}B(a)} \to I$. If $B(a)$ is $\Pi$-compact for every
  $a : A$ and the fiber of $t$ over $i$ has decidable equality for every
  $i : I$, then $\mathsf{W}_{s,t}(i)$ also has decidable equality for every
  $i : I$.
\end{therm}

Let us see how to obtain Theorem~\ref{indexedWtypedeceq'} from
Theorem~\ref{indexedWtypedeceq}.
\begin{proof}[Proof of Theorem~\ref{indexedWtypedeceq'} (using Theorem~\ref{indexedWtypedeceq}) \textnormal{(\coqident{MoreFoundations}{Wtypes}{indexedWtype_deceq'})}]
  Suppose that $A$ has decidable equality and $I$ is a set. We are to show that
  the fiber of $t$ over $i$ has decidable equality for every $i : I$. Let
  $i : I$ be arbitrary. Suppose we have $(a,p)$ and $(a',p')$ in the fiber of
  $t$ over $i$. Since $A$ has decidable equality, we can decide whether $a$ and
  $a'$ are equal or not. If they are not, then certainly $(a,p) \neq
  (a',p')$. If they are, then we claim that the dependent pairs $(a,p)$ and
  $(a',p')$ are also equal. If $e : a = a'$ is the supposed equality, then it
  suffices to show that $\transport^{\lambda x : A.t(x) = i}(e,p) = p'$,
  but both these terms are paths in $I$ and $I$ is a set, so they must be equal.
\end{proof}

We now embark on a proof of Theorem~\ref{indexedWtypedeceq}. For the remainder of this
section, let us fix types $A$ and $I$, a type family $B$ over $A$ and maps
$t : A \to I$ and $s : \pa*{\sum_{a : A} B(a)} \to I$.

We do not prove the theorem directly. The statement makes it
impossible to assume two elements $u,v : \mathsf{W}_{s,t}(i)$ and proceed by
induction on \emph{both} $u$ and $v$. Instead, we will state and prove a more
general result that is amenable to a proof by induction. But first, we need more
general lemmas and some definitions.

\begin{lemma}
  \label{rightpairinj}
  Let $X$ be a type and let $Y$ be a type family over it. If $X$ is a set, then
  the right pair function is injective, in the following sense: if
  $(x,y) = (x,y')$ as terms of $\sum_{a : X}Y(a)$, then $y = y'$.
\end{lemma}
\begin{proof}[\coqidentproof{MoreFoundations}{Wtypes}{dec_depeq}]
  Suppose $X$ is a set, $x : X$ and $y,y' : Y(x)$ with $e : (x,y) =
  (x,y')$. From $e$, we obtain $e_1 : x = x$ and
  $e_2 : \transport^{Y}(e_1,y) = y'$. Since $X$ is a set, we must have
  that $e_1 = \mathsf{refl}_{x}$, so that from $e_2$ we obtain a term of type
  $y \equiv \transport^Y(\mathsf{refl}_{x},y) = y'$, as desired.
\end{proof}

\begin{definition}[\coqident{MoreFoundations}{Wtypes}{subtrees}]
  For each $i : I$, define
  \[
    \mathsf{sub}_i : \mathsf{W}_{s,t} (i) \to \sum_{p:\mathsf{fib}_t(i)}
    \prod_{b : B(\fst(p))}\mathsf{W}_{s,t}( s(\fst(p),b))
  \]
  by induction:
  \[
    \mathsf{sub}_{t(a)}(\mathsf{indexedsup}(a,f)) \colonequiv
    \pa*{\pa*{a,\mathsf{refl}_{t(a)}},f}.
  \]
  For notational convenience, we will omit the subscript of $\mathsf{sub}$.
\end{definition}

\begin{lemma}
  \label{eqtosubtreeeq}
  Let $a : A$ and $f,g : \prod_{b : B(a)} \mathsf{W}_{s,t}(s(a,b))$. If the
  fiber of $t$ over $i$ has decidable equality for every $i : I$, then
  $\mathsf{indexedsup}(a,f) = \mathsf{indexedsup}(a,g)$ implies $f = g$.
\end{lemma}
\begin{proof}[\coqidentproof{MoreFoundations}{Wtypes}{subtrees_eq}]
  Suppose $\mathsf{indexesup}(a,f) = \mathsf{indexedsup}(a,g)$. Then
  \[
    \pa*{\pa*{a,\mathsf{refl}_{t(a)}},f} \equiv
    \mathsf{sub}(\mathsf{indexedsup}(a,f)) =
    \mathsf{sub}(\mathsf{indexedsup}(a,g)) \equiv
    \pa*{\pa*{a,\mathsf{refl}_{t(a)}},g}.
  \]
  As \(\mathsf{fib}_t(i)\) is decidable, it is a set by
  Hedberg's~Theorem~\cite[Theorem~7.2.5]{HoTTbook}. Therefore \(f = g\) by
  Lemma~\ref{rightpairinj}.
\end{proof}

\begin{definition}[\coqident{MoreFoundations}{Wtypes}{getfib}]
  \label{getfib}
  For every $i : I$, define a function
  $\mathsf{getfib}_i : \mathsf{W}_{s,t}(i) \to \mathsf{fib}_t(i)$ inductively by
  \[
    \mathsf{getfib}_{t(a)}(\mathsf{indexedsup}(a,f)) \colonequiv
    (a,\mathsf{refl}_{t(a)}).
  \]
  In future use, we omit the subscript of $\mathsf{getfib}$.
\end{definition}
\begin{lemma}
  \label{getfibtransport}
  Let $i,j : I$ with a path $p : i = j$ and $w : \mathsf{W}_{s,t}(i)$. We have
  the following equality:
  \[
    \mathsf{getfib}(\transport^{\mathsf{W}_{s,t}}(p,w)) =
    (\fst(\mathsf{getfib}(w)),\snd(\mathsf{getfib}(w)) \pathcomp p).
  \]
\end{lemma}
\begin{proof}[\coqidentproof{MoreFoundations}{Wtypes}{getfib_transport}]
  By path induction on $p$.
\end{proof}

We are now in position to state and prove the lemma from which
Theorem~\ref{indexedWtypedeceq} follows.
\begin{lemma}
  \label{indexedWtypedeceqtransport}
  Suppose that $B(a)$ is $\Pi$-compact for every $a :A $ and that the fiber of
  $t$ over each $i : I$ has decidable equality. For any $i : I$,
  $u : \mathsf{W}_{s,t}(i)$, $j : I$, path $p : i = j$ and
  $v : \mathsf{W}_{s,t}(j)$, the type
  \[
    \transport^{\mathsf{W}_{s,t}}(p,u) = v
  \]
  is decidable.
\end{lemma}
\begin{proof}[\coqidentproof{MoreFoundations}{Wtypes}{indexedWtype_deceq_transport}]
  Suppose $i : I$ and $u : \mathsf{W}_{s,t}(i)$. We proceed by induction on $u$
  and so we assume that $u \equiv \mathsf{indexedsup}(a,f)$. The induction
  hypothesis reads:
  \begin{equation*}\label{IH}
    \prod_{b : B(a)}\prod_{j' : I}\prod_{p' : s(a,b) = j'}\prod_{v' :
      \mathsf{W}_{s,t}(j')} (\transport^{\mathsf{W}_{s,t}}(p',f(b))=v')
    \text{ is decidable }. \tag{$\ast$}
  \end{equation*}
  Suppose we have $j : I$ with path $p : t(a) = j$ and $v : \mathsf{W}_{s,t}(j)$.
  By induction, we may assume that $v \equiv \mathsf{indexedsup}(a',f')$. We are
  tasked to show that
  \begin{equation*}\label{indexedsupdec}
    \transport^{\mathsf{W}_{s,t}}(p, \mathsf{indexedsup}(a,f)) =
    \mathsf{indexedsup}(a',f') \tag{$\dagger$}
  \end{equation*}
  is decidable, where $p : t(a) = t(a')$.

  By assumption the fiber of $t$ over $t(a')$ has decidable equality. Hence, we
  can decide if $\pa*{a',\mathsf{refl}_{t(a')}}$ and $(a,p)$ are equal or
  not. Suppose first that the pairs are not equal. We claim that in this case
  $\lnot \eqref{indexedsupdec}$. For suppose we had $e : \eqref{indexedsupdec}$,
  then
  \[
    \mathsf{ap}_{\mathsf{getfib}}(e) :
    \mathsf{getfib}(\transport^{\mathsf{W}_{s,t}}(p,
    \mathsf{indexedsup}(a,f))) = \mathsf{getfib}(\mathsf{indexedsup}(a',f')).
  \]
  By definition, the right hand side is $(a',\mathsf{refl}_{t(a')})$. By
  Lemma~\ref{getfibtransport}, the left hand side is equal to
  $(a, \mathsf{refl}_{t(a)} \pathcomp p)$ which is in turn equal to $(a,p)$,
  contradicting our assumption that $\pa*{a',\mathsf{refl}_{t(a')}}$ and $(a,p)$
  were not equal.

  Now suppose that $\pa*{a',\mathsf{refl}_{t(a')}} = (a,p)$. From this, we
  obtain paths $e_1 : a' = a$ and
  $e_2~:~\transport^{\lambda x : A.t(x) =
    t(a')}\pa*{e_1,\mathsf{refl}_{t(a')}} = p$. By path induction, we may assume
  $e_1 \equiv \mathsf{refl}_{a'}$, so that from $e_2$ we obtain a path
  \[
    \rho : \mathsf{refl}_{t(a')} = p.
  \]
  Using this path, we see that the left hand side of \eqref{indexedsupdec} is
  equal to $\mathsf{indexedsup}(a',f)$, so we are left to show that
  \[
    \mathsf{indexedsup}(a',f) = \mathsf{indexedsup}(a',f')
  \]
  is decidable.

  By induction hypothesis \eqref{IH} and the fact that $a \equiv a'$, the type
  $ f(b) = f'(b) $ is decidable for every $b : B(a')$. Since $B(a')$ is
  $\Pi$-compact, this implies that $\prod_{b : B(a')}f(b) = f'(b)$ is decidable.

  Suppose first that $\prod_{b : B(a')}f(b) = f'(b)$. Function extensionality
  then yields $f = f'$, so that
  $\mathsf{indexedsup}(a',f) = \mathsf{indexedsup}(a',f')$.

  On the other hand, suppose $\lnot\prod_{b : B(a')}f(b) = f'(b)$. We claim that
  then, $\mathsf{indexedsup}(a',f)$ cannot be equal to
  $\mathsf{indexedsup}(a',f')$. For suppose that
  $\mathsf{indexedsup}(a',f) = \mathsf{indexedsup}(a',f')$. Then
  Lemma~\ref{eqtosubtreeeq} yields $f = f'$, contradicting our assumption that
  $\lnot\prod_{b : B(a)}f(b) = f'(b)$, and finishing the proof.
\end{proof}

\begin{proof}[Proof of Theorem~\ref{indexedWtypedeceq} \textnormal{(\coqident{MoreFoundations}{Wtypes}{indexedWtype_deceq})}]
  Let $i : I$ and $u,v : \mathsf{W}_{s,t}(i)$. Taking $j \colonequiv i$ and
  $p \colonequiv \mathsf{refl}_i$ in Lemma~\ref{indexedWtypedeceqtransport}, we
  see that $u = v$ is decidable, as desired.
\end{proof}

\section{Size matters}\label{sizematters}
In this penultimate section, we explain some of the subtleties regarding dcpos
and universe levels. In particular, we revisit the dcpo of continuous functions
while rigorously keeping track of universe levels. In the end, our analysis
shows that, even in the absence of propositional resizing, the interpretation
function $\densem{-}$ of the Scott model is well-defined
(Theorem~\ref{interpretationtyped}). (For more on predicative domain theory, the reader
may wish to consult our recent work~\cite{deJongEscardo2021a,deJongEscardo2021b}.)

As mentioned in the introduction, our results are formalised in Agda
\cite[\texttt{PCFModules}]{TypeTopology}.

To study universe levels, let us suppose that we have a tower of type universes
$\mathcal U_0 : \mathcal U_1 : \dots$, indexed by meta natural numbers. (In the
end, it will turn out that having just two universes
$\mathcal U_0 : \mathcal U_1$ is sufficient for our purposes.) Let us fix some
notation for (raising) universe levels. We write $\mathcal U_i^+$ for
$\mathcal U_{i+1}$ and $\mathcal U_i \sqcup \mathcal U_j$ for
$U_{\max(i,j)}$. The universes are assumed to be closed under $+$-, $\Sigma$-
and $\Pi$-types and if $X : \mathcal U$ and $Y : X \to \mathcal V$, then
$\sum_{x : X}Y(x),\prod_{x : X}Y(x) : \mathcal U \sqcup \mathcal V$. Finally,
since $\mathcal U : \mathcal U^+$, we have
$\sum_{X : \mathcal U}Y(X) : \mathcal U^+ \sqcup \mathcal V$ if
$Y : \mathcal U \to \mathcal V$.

\subsection{The lifting}
In Section~\ref{sec:overview}, we introduced $\Omega$ as the type of propositions in
the universe $\mathcal U_0$. To see why we made this particular choice of type
universe and to appreciate the considerations involved, it is helpful to
consider a more general situation. Let us write $\Omega_{\mathcal T}$ for the
propositions in some type universe $\mathcal T$. Define the (generalised)
lifting $\lift_{\mathcal T}(X)$ of a type $X$ is as
$\lift_{\mathcal T}(X) \colonequiv \sum_{P : \Prop_{\mathcal {T}}} (P \to X)$.

Now observe that if $X$ is a type in a universe $\mathcal U$, then lifting
(potentially) raises the universe level, as $\lift_{\mathcal T}(X)$ is a type in
universe $\mathcal{T}^+ \sqcup \mathcal U$. However, if $X$ happens to be a type
in $\mathcal T^+$, then $\lift_{\mathcal T} (X)$ also lives in $\mathcal {T}^+$.
Moreover, repeated applications of $\lift$ do not raise the universe level any
further, because if $X$ is in $\mathcal T^+ \sqcup \mathcal U$, then
$\lift_{\mathcal T}(X)$ is as well. Despite the fact that lifting raises the
universe level, one can write down the monad laws for $\lift_{\mathcal T}$ and
they typecheck.

Let $X$ and $I$ be types in universes $\mathcal U$ and $\mathcal V$,
respectively. Suppose that $u : I \to \lift_{\mathcal T} (X)$. Note that
$\sum_{i : I} \isdefined(u_i)$ is in $\mathcal V \sqcup \mathcal T$. When
considering $\lift_{\mathcal T}(X)$ as a dcpo
(cf. Theorem~\ref{liftofsetisdcpo}), we want $\sum_{i : I} \isdefined(u_i)$ to
be in $\mathcal T$ again. One way to ensure this, is to take $\mathcal V$ to be
$\mathcal U_0$. This would make $\lift_{\mathcal T}(X)$ a $\mathcal
U_0$-dcpo. Indeed, this is what we prove in the Agda formalisation. In
particular, this means that $\lift_{\mathcal T}(X)$ has $\natt$-indexed directed
suprema, which suffices for the Scott model of PCF.

\subsection{The dcpo of continuous functions}

In fact, we should be even more precise when it comes universe levels and dcpos
than we have been so far. Write
$\mathcal W\text{-}\pa*{\mathsf{DCPO}_\bot}_{\mathcal U,\mathcal V}$ for the
type of $\mathcal W$-directed complete posets with a least element whose
underlying type is in $\mathcal U$ and whose underlying order takes values in
$\mathcal V$.

Then $\lift_{\mathcal U_0}(\natt) \equiv \lift(\natt)$ is of type
$\mathcal U_0\text{-}\pa*{\mathsf{DCPO}_\bot}_{\mathcal U_1,\mathcal U_1}$, for
example. (One easily checks that the order $\sqsubseteq$ from
Theorem~\ref{liftofsetisdcpo} has values in $\mathcal U_1$.)

Recall that
$\densem{\sigma \Rightarrow \tau} \equiv \densem{\tau}^{\densem{\sigma}}$, the
dcpo with $\bot$ of continuous functions from $\densem{\sigma}$ to
$\densem{\tau}$, so let us investigate the universe levels surrounding the
exponential. In general, we have:
\begin{equation*}\label{universes-exponential}
  \text{if } \mathcal D : \mathcal W\text{-}\pa*{\mathsf{DCPO}_\bot}_{\mathcal
    U,\mathcal V} \text{ and } \mathcal E : \mathcal
  W\text{-}\pa*{\mathsf{DCPO}_\bot}_{\mathcal U',\mathcal V'}, \text{then }
  {\mathcal E}^{\mathcal D} : \mathcal
  W\text{-}\pa*{\mathsf{DCPO}_\bot}_{\mathcal W^+\sqcup\mathcal V\sqcup \mathcal
    V'\sqcup \mathcal U\sqcup\mathcal U',\,\mathcal U\sqcup\mathcal
    V'}. \tag{$\dagger$}
\end{equation*}
We explain the universe levels involved as follows.

Let $\mathcal D$ be of type
$\mathcal W\text{-}\pa*{\mathsf{DCPO}_\bot}_{\mathcal U,\mathcal V}$ and write
$D$ and $\leq_{\mathcal D}$ for its underlying type and order,
respectively. Further, let $\mathcal E$ be of type
$\mathcal W\text{-}\pa*{\mathsf{DCPO}_\bot}_{\mathcal U',\mathcal V'}$ and write
$E$ and $\leq_{\mathcal E}$ for its underlying type and order, respectively.

The underlying type of the exponential $\mathcal E^{\mathcal D}$ is the type of
functions from $D$ to $E$ that are continuous. The underlying order is the
pointwise order: if $f$ and $g$ are continuous functions from $D$ to $E$, then
$f \leq_{\mathcal{E^D}} g$ if $\prod_{x : D} f(x) \leq_{\mathcal E} g(x)$.

Because $D$ is in $\mathcal U$ and $\leq_{\mathcal E}$ takes values in
$\mathcal V'$, we see that $\leq_{\mathcal{E^D}}$ takes values in
$\mathcal U\sqcup \mathcal V'$.

Furthermore, the type of functions from $D$ to $E$ is in
$\mathcal U\sqcup \mathcal U'$. But the type of \emph{continuous} functions also
mentions $\leq_{\mathcal D}$ and $\leq_{\mathcal E}$ and \emph{all} directed
families indexed by a type in $\mathcal W$. In particular, the latter means that
the definition of the type of continuous functions
contains~$\prod_{I : \mathcal W}$. Therefore the type of continuous functions is
in
$\mathcal W^+\sqcup \mathcal V\sqcup\mathcal V'\sqcup\mathcal U\sqcup \mathcal
U'$.

\subsection{The Scott model of PCF}
Given the increasing universe levels in \eqref{universes-exponential}, one might
ask if there can be universes $\mathcal U,\mathcal V,\mathcal W$ such that
\[
  \densem{-} : \text{PCF types} \to \mathcal
  U\text{-}\pa*{\mathsf{DCPO}_\bot}_{\mathcal V,\mathcal W}
\]
typechecks.

As we mentioned,
$\lift_{\mathcal U_0}(\natt) \equiv \lift (\natt) : \mathcal
U_0\text{-}\pa*{\mathsf{DCPO}_\bot}_{\mathcal U_1,\mathcal U_1}$. Since,
$\densem{\iota} \equiv \lift(\natt)$, one would hope that
\[
  \densem{-} : \text{PCF types} \to \mathcal
  U_0\text{-}\pa*{\mathsf{DCPO}_\bot}_{\mathcal U_1,\mathcal U_1}.
\]
And indeed, this is the case.

\begin{therm}\label{interpretationtyped}
  The interpretation function $\densem{-}$ from PCF types to dcpos with $\bot$
  can be typed~as:
  \[
    \densem{-} : \textup{PCF types} \to \mathcal
    U_0\text{-}\pa*{\mathsf{DCPO}_\bot}_{\mathcal U_1,\mathcal U_1}.
  \]
\end{therm}
\begin{proof}
  If, in \eqref{universes-exponential}, we take $\mathcal W$ to be
  $\mathcal U_0$ and $\mathcal U,\mathcal U',\mathcal V,\mathcal V'$ all to be
  $\mathcal U_1$, then \eqref{universes-exponential} reads:
  \[
    (-)^{(-)} : \mathcal U_0\text{-}\pa*{\mathsf{DCPO}_\bot}_{\mathcal
      U_1,\mathcal U_1} \to \mathcal
    U_0\text{-}\pa*{\mathsf{DCPO}_\bot}_{\mathcal U_1,\mathcal U_1} \to \mathcal
    U_0\text{-}\pa*{\mathsf{DCPO}_\bot}_{\mathcal U_1,\mathcal U_1},
  \]
  as desired.
\end{proof}

\section{Conclusion and future work}\label{conclusion}
Our development confirms that univalent type theory is well adapted to the
constructive formalisation of domain-theoretic denotational semantics of
programming languages like PCF, which was the original goal of this
investigation. Moreover, our development is predicative. In particular, we have
given a predicative version of directed complete posets. Our results show that
partiality in univalent type theory via lifting works well. We rely crucially on
Voevodsky's treatment of subsingletons as truth values. In particular, the
propositional truncation plays a fundamental and interesting role in this
work. Finally, we saw an interesting application of the abstract theory of
indexed $\mathsf{W}$-types in characterising the propositions that come from PCF
terms of the base type.

Regarding the Scott model of PCF, there are two questions for future research:
\begin{enumerate}
\item Is there a natural extension of the map
  $\densem{\iota} \xrightarrow{\fst} \Omega$ to all PCF types? Can we
  characterise the propositions at types other than $\iota$, e.g.\ the
  propositions at type $\iota \Rightarrow \iota$? Are they still semidecidable?
\item How can we better understand the fact that only semidecidable
  propositions occur for the Scott model, but that restricting to such
  propositions somehow needs a weak form of countable choice?
\end{enumerate}
In \cite{deJongEscardo2021a} we develop domain theory further in predicative and
constructive univalent type theory, including continuous and algebraic dcpos,
ideal completions and Scott's famous \(D_\infty\). Complementing this work, the
paper~\cite{deJongEscardo2021b} explores some aspects of domain theory that
cannot be done predicatively.

\bibliographystyle{plain}
\bibliography{references}

\end{document}